\documentclass[11pt,reqno, oneside, english]{amsart}

\usepackage{amsmath,amsthm,amssymb,amsfonts}
\usepackage{pdfsync}
\usepackage{mathtools}
\usepackage[colorlinks=true, linktocpage=true]{hyperref}
\hypersetup{citecolor=blue}
\usepackage{booktabs, comment}
\usepackage{todonotes}
\usepackage{bm}
\usepackage{stmaryrd}

\usepackage[all]{xy}
\usepackage{pdfsync}
\usepackage{amssymb}
\usepackage{color}
\usepackage{longtable}
\usepackage{stackengine}
\def\subrangle#1{\stackengine{5pt}{}{$\!\scriptstyle #1$}{U}{l}{F}{F}{L}}

\usepackage[margin=1.3 in]{geometry}

\newcommand{\norm}[1]{\left\lVert#1\right\rVert}

\newcommand{\T}{{\bf \Theta}}
\newcommand{\A}{\mathbb{A}}
\newcommand{\s}{\mathbb{S}}
\newcommand{\C}{\mathbb{C}}
\newcommand{\V}{\mathbb{V}}

\newcommand{\Heis}{\mathsf{H}}
\newcommand{\End}{\textnormal{End}}

\newcommand{\G}{\mathbb{G}}

\newtheorem{theorem}{Theorem}[section]
\newtheorem{theorem*}{Theorem}
\newtheorem{corollary*}[theorem*]{Corollary}
\newtheorem{corollary}[theorem]{Corollary}
\newtheorem{lemma}[theorem]{Lemma}         
\newtheorem*{lemma*}{Lemma}         
\newtheorem{proposition}[theorem]{Proposition}
\newtheorem*{proposition*}{Proposition}

\theoremstyle{definition}
\newtheorem{remark}[theorem]{Remark}
\newtheorem{definition}[theorem]{Definition}              
\numberwithin{equation}{section}

\title[Theta correspondence via $C^*$-algebras]{Theta correspondence via group $C^*$-algebras}

\author[Goffeng]{Magnus Goffeng}
\address{\normalfont{Centre for Mathematical Sciences, Lund University \\
Box 118, SE-221 00 Lund, Sweden}}
\email{magnus.goffeng@math.lth.se}

\author[Mesland]{Bram Mesland}
\address{\normalfont{Mathematisch Instituut, Universiteit Leiden \\
Postbus 9512, 2300 RA Leiden, Netherlands}}
\email{b.mesland@math.leidenuniv.nl}

\author[\c{S}eng\"un]{Mehmet Haluk \c{S}eng\"un}
\address{\normalfont{School of Mathematical and Physical Sciences\\
University of Sheffield,
Hounsfield Road, Sheffield, S3 7RH, UK}}
\email{m.sengun@sheffield.ac.uk}

\begin{document}

\begin{abstract} 
We prove that the well-known explicit construction of the local theta correspondence by Li has a simple interpretation in terms of group $C^*$-algebras. In particular, we deduce that in two standard cases where Li's method work, local theta correspondence arises from a continuous functor.  Moreover, using results from a companion paper, we treat global theta correspondence using $C^*$-algebraic methods. As a byproduct, we exhibit that Rallis inner product formula can be interpreted as a certain natural inclusion being an isometry. 
\end{abstract}

\maketitle
\tableofcontents
\allowdisplaybreaks

\section{Introduction}
Let $(G',G)$ be a \emph{dual pair}, that is, a pair of reductive subgroups of a symplectic group ${\rm Sp}(W)$ such that they are each others' centralizers. Following Roger Howe, we say that an admissible irreducible representation $\pi$ of $G$ maps to an admissible irreducible representation $\theta(\pi)$ of $G'$ if the $G'\times G$-representation $\theta(\pi) \otimes \pi$ is a quotient of a certain distinguished representation of the metaplectic cover of ${\rm Sp}(W)$ called the \emph{oscillator representation}. Theta correspondence, also known as Howe duality, is the fact that the rule $\pi \mapsto \theta(\pi)$ sets up a bijection between subsets of admissible duals of $G'$ and $G$. This correspondence can be set up both over local fields as well as globally -- over the adeles of a number field. 

Since its inception in the mid-70s, theta correspondence became a fundamental area of research in both representation theory and the theory of automorphic forms. In this paper, we will place both the local and the global theta correspondence in the context of $C^*$-algebras and Rieffel induction which opens up the study of theta correspondence potentially to new tools and techniques. 

Key to our approach is an explicit construction  of $\theta(\pi)$, discovered by Jian-Shu Li in \cite{Li-89},  for unitary representations $\pi$ of $G$ when, roughly put, $G$ was much smaller than $G'$. Using his explicit construction, Li was able to prove that in this case, all unitary representations $\pi$ of $G$ entered the theta correspondence and that $\theta(\pi)$ was also unitary. Thus the theta correspondence gives rise to an embedding of unitary duals $\widehat{G} \hookrightarrow \widehat{G'}$. Over the years, the explicit construction of Li has been extended and modified by various researchers. For example, it is known that for a general dual pair, Li's method can be used to prove that if a tempered representation $\pi$ of the smaller group enters the theta correspondence, then $\theta(\pi)$ is unitary. 

In our paper, we show that Li's method is simply an instance of the general induction procedure for representations of $C^*$-algebras introduced by Marc Rieffel \cite{Rieffel-74} in the early 1970's. We will call this {\em Rieffel induction} for convenience. 

Let briefly describe the Rieffel induction procedure, while leaving the details for later. Consider two $C^*$-algebras $A$ and $B$. Say we have a right {\em Hilbert $C^*$-module} $X$ over $B$, meaning $X$ is a right $B$-module carrying a $B$-valued inner product similar to a Hilbert space for $B=\C$. If $X$ is equipped with a left action $\alpha:A\to \End_B^*(X)$ as adjointable $B$-linear operators on $X$ then $X$ is called an $(A,B)${\em -correspondence}. Such a correspondence $X$ induces a continuous tensor product functor 
$${\rm Ind}_B^A(X) : (\pi,\mathcal{H}_\pi)\mapsto (\alpha\otimes_B 1, X\otimes_B \mathcal{H}_\pi),$$
from the representations of $B$ to the representations of $A$ via internal tensor product. 

\subsection{A blueprint} \label{blueprint}
Before we present our main results, let us give a schematic overview of how we employ Rieffel induction to capture theta correspondence. 

Recall that given a locally compact group, one can associate a $C^*$-algebra $C^*(G)$ (called the {\it maximal $C^*$-algebra of $G$}) which captures the unitary representations of $G$. By taking a suitable quotient, one obtains the $C^*$-algebra $C^*_r(G)$ (called the {\it reduced $C^*$-algebra of $G$}) that captures the unitary representations of $G$ weakly contained in the regular representation of $G$. Since we will be dealing with reductive groups exclusively, we will refer to such representations simply as \emph{tempered} representations.   

Say we are given two groups $G',G$ with an isometric action of the product group $G' \times G$ on an inner product space $X_0$. Given $\xi,\eta\in X_0$, let us denote the associated matrix coefficient function on $G$ by
$$\langle \xi,\eta\rangle\subrangle{G}:=[g\mapsto \langle \xi,g.\eta\rangle\subrangle{X_0}]$$
Assume that these matrix coefficient functions satisfy the following
\begin{itemize} \label{innerjkednad}
\item[(A1)] $ \langle \xi,\eta\rangle\subrangle{G} \in C^*(G)$,
\item[(A2)] $\langle \xi,\xi \rangle\subrangle{G} \geq 0$  
\end{itemize}
where positivity is meant in the sense of $C^*$-algebras. Then there is an associated $C^*(G)$-valued inner product in which we can complete $X_0$ to a Hilbert $C^*$-module $X$ over $C^*(G)$. The action of $G'$ on $X_0$ leads to a $(C^*(G'),C^*(G))$-correspondence structure on $X$, which in turn defines a Rieffel induction functor ${\rm Ind}_{G}^{G'}(X)$ from unitary representations of $G$ to unitary representations of $G'$. If in (A1) we assumed the matrix coefficients to land in $C^*_r(G)$ then we would end up with a functor from the tempered representations of $G$ to the unitary representations of $G'$. 

If $(\pi, \mathcal{H}_\pi)$ is a unitary representation of $G$, then ${\rm Ind}_{G}^{G'}(X,\pi)$ is a unitary representation of $G'$ afforded on the space $X\otimes_B \mathcal{H}_\pi$ with the inner-product
\begin{equation} \label{Li-resemble}
\langle \xi \otimes v , \eta \otimes w \rangle := \int_G  \langle \xi,g.\eta\rangle\subrangle{X_0}  \langle v ,g.w \rangle\subrangle{\mathcal{H}_\pi} dg
\end{equation}
for $\xi,\eta \in X_0$ and $v,w \in \mathcal{H}_\pi$. Readers who are familiar with Li's work will surely recognise that this is the integral of matrix coefficients that lies at the centre of Li's method! Assumption (A1) ensures that the integral in (\ref{Li-resemble}) converges for any $\pi$, whereas (A2) ensures that the form defined in (\ref{Li-resemble}) is positive semi-definite. These issues of convergence and positivity are key to making Li's method work.

\subsection{Main results}
The connection between local theta correspondence and Rieffel induction was first described in \cite{mesland-sengun-ER}. However, the paper \cite{mesland-sengun-ER} only treated the very special case of equal rank dual pairs over non-archimedean local fields (where one could explain various extra features that the theta correspondence enjoyed using $C^*$-algebraic Morita equivalence machinery). Our treatment in this paper is much more general in scope. Moreover, we also treat global theta correspondence using results from our recent paper \cite{GMS-24}. 

\subsubsection{Local results} 
Consider a dual pair $(G',G)$ over a local field. Let $\s$ denote the subspace of smooth vectors of the Hilbert space that affords the oscillator representation of $G'\times G$. The basic idea is to follow the blueprint provided in Section \ref{blueprint} above, taking $X_0$ to be $\s$. The key challenge is the verification of the assumptions (A1) and (A2). As we signalled already, this is directly related to Li's method. In fact, verification of assumptions (A1) and (A2) is more or less equivalent to proving that Li's method works! Indeed, we verify these assumptions in two cases where  Li's construction is known to work; one is the theta lifting of tempered irreducible representations of the smaller group in a general type I dual pair, and the other is Li's original case, that is, theta lifting of unitary irreducible representations of the smaller group in a stable range dual pair. 

The following is Theorem \ref{main-result-one} of the paper.
\begin{theorem*}
\label{thm1}
Let $(G',G)$ be a Type I dual pair over a local field with $G$ the smaller group. The completion $\T^{r}$ of the $G'\times G$-space $\s$ as a $C^*_r(G)$-Hilbert module with the inner product \eqref{innerjkednad} defines a 
$(C^*(G'),C^*_r(G))$-correspondence. The associated Rieffel induction functor ${\rm Ind}_{G}^{G'}(\T^r)$ captures the theta lifting of tempered irreducible representations of $G$. More precisely, if $\pi$ is an irreducible tempered representation of $G$ then we have an isomorphism
$$\theta(\pi^*) \simeq {\rm Ind}_{G}^{G'}(\T^r, \pi)$$
of $G'$-representations. 
\end{theorem*}

We also prove the analogue of the above theorem for all unitary representations and dual pairs in the \emph{stable range}, see Theorem \ref{main-result-two}. We discuss the connection of Theorem \ref{thm1} to recent work on the \emph{big theta lift} by Loke--Przebinda \cite{Loke-Przebinda}, and compare the big theta lift of $\pi$ to $\theta(\pi)$ in Subsection \ref{bigh-theta}.

\subsubsection{Global results} Theorem \ref{thm1} shows that the local theta correspondence is implemented by Rieffel induction. By using results from the companion paper \cite{GMS-24} we show that the same holds for the global theta correspondence. For a number field $F$ with adele ring $\A$, let $\s_\A$ denote the space of smooth vectors of the (global) oscillator representation of $G'(\A)\times G(\A)$. The following can be found in the paper as Theorem \ref{main-result-four}.

\begin{theorem*}
\label{thm2}
Let $(G',G)$ be a Type I dual pair over a number field $F$ with $G$ the smaller group. The local $C^*$-correspondences in the previous theorem can be bundled up to give a $( C^*(G'(\A)),C^*_r(G(\A)) )$-correspondence $\T^{r}_\A$. The associated Rieffel induction functor ${\rm Ind}_{G(\A)}^{G'(\A)}(\T^r_\A)$ that captures the theta lifting of tempered irreducible representations of $G(\A)$. More precisely, if $\pi$ is an irreducible tempered representation of $G(\A)$ then we have an isomorphism
$$\theta(\pi^*) \simeq {\rm Ind}_{G(\A)}^{G'(\A)}(\T^r_\A, \pi)$$
of $G'(\A)$-representations. 
\end{theorem*}

We also prove the analogue of the above theorem for all unitary representations and dual pairs in the stable range, see Theorem \ref{main-result-two}.

Finally, we present another viewpoint brought on by the result that theta correspondence is a Rieffel induction; namely a rephrasing of Rallis celebrated inner product formula \cite{Rallis-84} (see Theorem \ref{RIPF} below). Let $(\pi,V_\pi)$ is a cuspidal automorphic representation of $G(\A)$, that is, $\pi$ is an irreducible subrepresentation of the regular representation of $G(\A)$ on $L^2_0(G_F \backslash G(\A))$. Given a cuspidal automorphic form $f \in V_\pi$ and $\phi \in \s_\A$, there is an associated automorphic form $\theta_\phi(f)$ on $G'(\A)$ defined by 
$$g \mapsto \int_{G_F \backslash G(\A)} \left( \sum_{x \in X_F} \omega(gh)(\phi)(x) \right )  f(h) dh.$$
The Rallis inner product formula calculates the $L^2$-inner-product 
$$\langle \theta_{\phi_1}(f_1), \theta_{\phi_2}(f_2) \rangle$$
when it is well-defined. Here $f_2,f_2 \in V_\pi$ and $\phi_1,\phi_2 \in \s_\A$. 

We show that the Rallis inner product formula is equivalent to the following result via an argument of a few lines. See 
Proposition \ref{rallisasisom}.

\begin{corollary*}
\label{cor3}
Let $(G',G)$ be a dual pair in the stable range over a number field $F$ with $G$ the smaller group. The $G'(\A)$-equivariant map
$$
\mathcal{Z} : \T_\A \otimes_{C^*(G(\A))} V_{\pi} \longrightarrow L^2_0(G'_F \backslash G'(\A)),$$
defined by extending the rule $\mathcal{Z} (\phi \otimes f ) :=  \theta_{\phi}(f)$ with linearity and density, is an isometric embedding.
\end{corollary*}

\subsection{Structure of the paper}
The paper is organized as follows. In Section \ref{Cstar}, we recall the relevant notions from $C^*$-algebra theory and Rieffel induction. The relevant notions from representation theory are recalled in Section \ref{sec:thetta}, where we review the oscillator representation, dual pairs, the theta correspondence and the necessary results of Li leading up to the results required to prove Theorem \ref{thm1}. 

Section \ref{sec:localda} contains the main construction of the local Hilbert $C^*$-module as phrased in Theorem \ref{thm1}. In the stable range case, we show in Subsection \ref{right-module} that a similar construction produces not only a $(C^*(G'),C^*_r(G))$-bimodule but in fact a $(C^*(G'),C^*(G))$-bimodule. In Subsection \ref{bigh-theta} we provide a novel characterization of when the big theta lift of $\pi$ coincides with $\theta(\pi)$ starting from recent work on the  big theta lift by Loke--Przebinda \cite{Loke-Przebinda}. We give a more precise description of the left action of $C^*(G')$ in Section \ref{local-Rallis}. In Theorem \ref{left-surject}, we show that over non-archimedean local fields when $(G',G)$ is an ortho-symplectic or unitary pair, the left action contains all $C^*_r(G)$-compact operators on the Hilbert module and in the stable range case, the left action contains all $C^*(G)$-compact operators. This has consequences for preserving type I representations. Finally, in Section \ref{globalmodule} we study the global situation leading up to the results summarized in Theorem \ref{thm2} and in Section \ref{sec:ralla} we study the connection to Rallis' inner product formula as described in Corollary \ref{cor3}.

\subsection{Acknowledgments}  Part of the research in this paper was carried out within the online research community ``Representation Theory and Noncommutative Geometry'' sponsored by the American Institute of Mathematics. The second author gratefully acknowledges the invaluable support of the EPSRC New Horizons grant EP/V049119/1 which provided the much needed research time to develop this project. We thank Hongyu He for a helpful correspondence and Tomasz Przebinda for his support and several very helpful conversations. 


\section{Background on $C^*$-algebras} 
\label{Cstar}
A $C^*$-algebra $A$ is an involutive Banach algebra whose involution $*$ satisfies the $C^*$-identity 
$$\norm{a^*a}=\norm{a}^{2},$$
for every $a \in A$. This identity leads to an intimate relationship between algebraic and topological properties of a $C^*$-algebra. As a result $C^*$-algebras form a rigid class of involutive Banach algebras enjoying a rich theory. 

If $\mathcal{H}$ is a Hilbert space, then the algebra $\mathbb{B}(\mathcal{H})$ of bounded operators on $\mathcal{H}$ is a $C^*$-algebra under the operator norm with taking adjoint as the involution. A representation of a $C^*$-algebra $A$ on a Hilbert space $\mathcal{H}$ is a homomorphism $A \to \mathbb{B}(\mathcal{H})$ of $C^*$-algebras. It is a foundational fact that every $C^*$-algebra admits an injective representation on a Hilbert space.

\subsection{The group $C^{*}$-algebra} Given a locally compact Hausdorff topological group $G$ with Haar measure, we let $C_c(G)$ denote the space of compactly supported functions. This is a $*$-algebra in the convolution product and involution given by
\begin{equation}
\label{eq: convolution-algebra}
f_{1}*f_{2}(t):=\int_{G}f_{1}(s)f_{2}(s^{-1}t)\mathrm{d}t,\,\, f^{*}(s):=\overline{f(s^{-1})}.
\end{equation} 
The $*$-algebra $C_{c}(G)$ is dense (in the $L^{1}$-norm) in the Banach $*$-algebra of integrable functions $L^{1}(G)$. 

It is well-known that there is a bijection between the unitary representations of $\widehat{G}$ of $G$ and the non-degenerate $*$-representations of $L^1(G)$: given a (strongly continuous) unitary representation $(\pi,V_\pi)$ of $G$ by unitary operators on a Hilbert space $V_{\pi}$, we obtain a $*$-representation of $L^1(G)$ (still denoted $\pi$) by integrating
\begin{equation}
\label{eq: integrated-form} 
\pi(f) := \int_G f(s) \pi(s) \mathrm{d}s,
\end{equation}
where $f \in L^1(G)$. The {\em (maximal) $C^*$-algebra} of $G$, denoted $C^*(G)$, is the enveloping $C^*$-algebra of $L^1(G)$, that is, we complete $L^1(G)$ with respect to the norm 
$$\norm{f}_{C^*(G)} := {\rm sup}_{\pi \in \widehat{G}} \norm{\pi(f)}.$$
The supremum is finite and $\norm{f}_{C^*(G)}\leq \norm{f}_{L^1(G)}$ since $*$-representations are contractive. As $*$-representations of $L^1(G)$ uniquely extend to $C^*(G)$, we have a bijection between unitary representations of $G$ and non-degenerate $*$-representations of $C^*(G)$. This bijection respects unitary equivalence and leads to an equivalence of categories
$${\rm URep}(G) \longrightarrow {\rm Rep}(C^*(G))$$
where ${\rm URep}(G)$ denotes the category of unitary representation of $G$.

As is well-known, the unitary dual $\widehat{G}$ can be topologised using the notion of {\em weak containment} which is based on uniform approximation of matrix coefficients on compact subsets of $G$. On the $C^*(G)$-representations side, weak containment has a very simple description: $\sigma$ is weakly contained in $\pi$ if and only if $\ker\pi$ is contained in $\ker\sigma$. The bijection between $\widehat{G}$ and the spectrum $\widehat{C^*(G)}$ of $C^*(G)$, arising from the the above categorical equivalence, is  a homeomorphism of topological spaces (see, e.g. Proposition 8.B.3 of \cite{Bekka-delaHarpe}):
$$\widehat{G} \xleftrightarrow{\text{homeom.}} \widehat{C^*(G)}.$$

We end this subsection by recalling the {\em reduced} $C^*$-algebra $C^*_r(G)$ of $G$. It is the quotient of $C^*(G)$ given by the image of the regular representation
$$\lambda : C^*(G) \longrightarrow \mathbb{B}(L^2(G))$$
obtained from integrating the regular representation of $G$ on $L^2(G)$. The reduced $C^{*}$-algebra sees all the {\em tempered} representations of $G$; these are the representations of $G$ which are weakly contained in the regular representation $\lambda$. We have a categorical equivalence 
$${\rm TRep}(G) \longrightarrow {\rm Rep}(C^*_r(G))$$
where ${\rm TRep}(G)$ denotes the category of tempered representation of $G$. Just like before, we have a homeomorphism
$$\widehat{G}_t \xleftrightarrow{\text{homeom.}} \widehat{C^*_r(G)}$$
between the sets of irreducibles on both sides.

\subsection{Hilbert $C^{*}$-modules} \label{Hilbert_modules}
In his foundational paper \cite{Rieffel-74}, Rieffel introduced a notion of induction for representations of $C^*$-algebras. In the case of $C^*$-algebras of groups, this induction theory captures and generalises Mackey's theory of induction for representations of groups. We will quickly summarise this theory in the setting of group $C^*$-algebras. All details can be found in the first sections of \cite{Raeburn-Williams}.

\begin{definition}
\label{def: Hilbert-module}
Let $B$ be a $C^*$-algebra. A \emph{Hilbert $C^{*}$-module} over $B$ is a complex vector space $X$ that is a right $B$-module equipped with a map $\langle {\cdot}, {\cdot} \rangle : X {\times} X \to B$ such that for $x,y\in X$, $b\in B$ and $\lambda,\mu\in\mathbb{C}$ we have
\begin{enumerate}
\item[(i)] $\langle x, \lambda y + \mu z \rangle = \lambda \langle x, y \rangle + \mu \langle x, z \rangle$,
\item[(ii)] $\langle x, y{\cdot}b \rangle = \langle x, y \rangle b$,
\item [(iii)] $\langle x,y \rangle^* = \langle y,x \rangle$, 
\item[(iv)] $\langle x,x \rangle =0$ if and only if $x=0$,
\item[(v)] $\langle x,x \rangle \geq 0$,
\item[(vi)] $X$ is complete with respect to the norm $||x||:=||\langle x,x \rangle ||^{1/2}$. 
\end{enumerate}
for all $x,y,z \in X$, $\lambda, \mu \in \C$ and $b \in B$. We say that $X$ is {\em full} if $\langle X,X \rangle$ spans a dense subspace in $B$.
\end{definition}
The following Lemma is useful in order to construct $C^{*}$-modules in practice.
\begin{lemma}[\cite{Raeburn-Williams}, \textrm{Lemma 2.16}] 
\label{lem: pre-inner-product}
Let $B$ be a $C^{*}$-algebra, $\mathcal{B}\subset B$ a dense $*$-subalgebra and $\mathcal{X}$ a right $\mathcal{B}$-module. Suppose there is a pairing
\[\mathcal{X}\times\mathcal{X}\to \mathcal{B}, \quad (x,y)\mapsto \langle x,y\rangle\]
satisfying conditions $(i)-(iii)$ of Definition \ref{def: Hilbert-module} and let
\[\mathcal{N}:=\left\{x\in\mathcal{X}:\langle x,x\rangle=0\right\}.\]
If for all $x\in\mathcal{X}$ we have that $\langle x,x\rangle\geq 0$ in the $C^{*}$-algebra $B$, then $\|x\|:=\|\langle x,x\rangle\|_{B}^{1/2}$ is a norm on $\mathcal{X}/\mathcal{N}$ and the completion $X$ of $\mathcal{X}/\mathcal{N}$ in this norm is a Hilbert $C^{*}$-module over $B$.
\end{lemma}

Now let $X$ be a Hilbert $C^{*}$-module over $B$. We let $\End_B ^*(X)$ denote the algebra of $B$-module endomorphisms $T:X\to X$ that are \emph{adjointable} in the sense that there exists $T^{*}:X\to X$ such that $\langle Tx,y\rangle=\langle x,T^{*}y\rangle$.  The vector space $\End_B^*(X)$ is a (proper) subspace of the space of bounded $B$-module endomorphisms and is a $C^*$-algebra under the operator norm that one derives from that on $X$. We denote by
\[\mathcal{U}(X)=\mathcal{U}(\End_B ^*(X)):=\{u\in \End_B ^*(X): u^{*}u=uu^{*}=1\},\]
the unitary group of $\End_B ^*(X)$. 

Given $x,y \in X$, we have the {\em rank one operator} $T_{x,y}$ defined by $T_{x,y}(z)=x \langle y,z\rangle$ for every $z \in X$. The linear span of the rank one operators form a two-sided ideal of $\End_B ^*(X)$ and the norm closure of this ideal is a $C^{*}$-algebra known as the ideal of {\em $B$-compact operators}\footnote{Note that these operators are not compact in the sense of Banach space theory.}, denoted $\mathbb{K}_B(X)$.
 
\subsection{Internal tensor product} \label{C*-correspondence}  
Let $X$ be a Hilbert $C^{*}$-module over $B$ and $\pi : B \to \mathbb{B}(V_{\pi})$ a $*$-representation of $B$ as bounded operators on a Hilbert space $V_{\pi}$. Consider the algebraic tensor product of vector spaces $X \otimes^{\mathrm{alg}} V_{\pi}$. By \cite[Proposition 4.5]{Lance} the right sesquilinear form 
\begin{equation} \label{localisation-inner-product} \langle x \otimes v, x' \otimes v'\rangle := \langle v, \pi(\langle x,x' \rangle)v' \rangle\subrangle{V_\pi},
\end{equation}
is positive and its radical 
$$N_{\pi}:=\left\{\xi \in X \otimes^{\mathrm{alg}} V_{\pi} \mid \langle \xi, \xi \rangle = 0 \right\}$$
is equal to the \emph{balancing subspace}
\begin{equation}
\label{balancing-ideal}
I_{\pi}:=\mathrm{span}\left\{xb \otimes v - x\otimes \pi(b)(v):x \in X, v \in V_{\pi}, b \in B\right\}.
\end{equation}
The completion of $(X \otimes^{\mathrm{alg}} V_{\pi})/N_\pi$ with respect to the inner product \eqref{localisation-inner-product} is a Hilbert space that we denote by $X \otimes_{B}V_{\pi}$ and is commonly called the {\em internal tensor product} of $X$ and $V_{\pi}$ over $B$. An operator $T\in\End_B ^*(X)$ induces an operator $T\otimes 1\in\mathbb{B}(X\otimes_{B}V_{\pi})$ via
\[(T\otimes 1)(x\otimes v):=Tx\otimes v.\]
This gives a $*$ homomorphism $\End_B ^*(X)\to \mathbb{B}(X\otimes_{B}V_{\pi})$. The following Lemma is well-known to experts but we include it for completeness, as it will play a crucial role in the proof of Theorem \ref{main-result-one}.

\begin{lemma}
\label{lem: dense subspace}
Let $X$ be a Hilbert $C^{*}$-module over a $C^{*}$-algebra $B$, $\pi:B\to \mathbb{B}(V_{\pi})$ a $*$-representation. Suppose that $\mathcal{X}\subset X$ and $\mathcal{V}_{\pi}\subset V_{\pi}$ are dense subspaces and let 
\begin{equation} \label{smooth-radical} 
\mathcal{N}_{\pi}:=N_{\pi}\cap (\mathcal{X}\otimes^{\textnormal{alg}}\mathcal{V}_{\pi})
\end{equation}
be the radical of the inner product \eqref{localisation-inner-product} restricted to $\mathcal{X}\otimes \mathcal{V}_{\pi}$. 
Then $(\mathcal{X}\otimes^{\textnormal{alg}} \mathcal{V}_{\pi})/\mathcal{N}_{\pi}$ is a nondegenerate inner product space and the map
\[(\mathcal{X}\otimes^{\textnormal{alg}} \mathcal{V}_{\pi})/\mathcal{N}_{\pi}\to X\otimes_{B} V_{\pi}, \quad [x\otimes v]_{\mathcal{N}_{\pi}} \mapsto x\otimes v,\]
is a well defined inner-product preserving injection with dense range. 
\end{lemma}

\begin{proof}
It is clear that $(\mathcal{X}\otimes^{\textnormal{alg}} \mathcal{V}_{\pi})/\mathcal{N}_{\pi}$ is a non-degenerate inner product space. Moreover the map 
\[(\mathcal{X}\otimes^{\textnormal{alg}} \mathcal{V}_{\pi})/\mathcal{N}_{\pi}\to (X\otimes^{\textnormal{alg}} V_{\pi})/N_{\pi}, \quad [x\otimes v]_{\mathcal{N}_{\pi}} \mapsto [x\otimes v]_{N_{\pi}},\]
is well defined and injective since $\mathcal{N}_{\pi}=N_{\pi}\cap (X\otimes^{\textnormal{alg}} V_{\pi})$, and clearly preserves the inner product. Since $(X\otimes^{\textnormal{alg}} V_{\pi})/N_{\pi}$ injects into $X\otimes_{B}V_{\pi}$, so does $(\mathcal{X}\otimes^{\textnormal{alg}} \mathcal{V}_{\pi})/\mathcal{N}_{\pi}$. To see that it has dense range, let $x\in X$ and $v\in V_{\pi}$.  Choose $y\in\mathcal{X}$ with $\|x-y\|_{X}\leq \frac{\varepsilon}{2\|v\|}$ and $w\in\mathcal{V}_{\pi}$ with $\|v-w\|\leq \frac{\varepsilon}{2\|y\|}$. Then, since the norm on $X\otimes_{B}V_{\pi}$ satisfies $\|x\otimes v\|\leq \|x\|\|v\|$ we have
\begin{align*}
\|x\otimes v-y\otimes w\|&\leq \|x\otimes v-y\otimes v\|+\|y\otimes v-y\otimes w\|\\
&=\|(x-y)\otimes v\|+\|y\otimes(v-w)\|\leq \|x-y\|\|v\|+\|y\|\|v-w\|< \varepsilon,
\end{align*}
which shows that elementary tensors $x\otimes v\in X\otimes_{B}V_{\pi}$ can be approximated by elementary tensors $y\otimes w$ with $y\in\mathcal{X}$ and $w\in\mathcal{V}_{\pi}$. Since the elementary tensors span a dense subspace, a standard argument shows that our map has dense range.
\end{proof}
This Lemma shows that the internal tensor product can be constructed using merely dense subspaces of the module and Hilbert space involved. The following Corollary presents an alternative viewpoint on this construction. It will be used in relation to big theta lifts as discussed below in Subsection \ref{bigh-theta}. 

In the setting of Lemma \ref{lem: dense subspace}, suppose further that $\mathcal{B}\subset B$ is a dense $*$-subalgebra and that $\mathcal{X}\cdot \mathcal{B}\subset \mathcal{X}$ and $\pi(\mathcal{B})\cdot \mathcal{V}_{\pi}\subset \mathcal{V}_{\pi}$. The balanced algebraic tensor product $\mathcal{X}\otimes_{\mathcal{B}}^{\mathrm{alg}}\mathcal{V}_{\pi}$ is the quotient $(\mathcal{X}\otimes_{\mathcal{B}}^{\mathrm{alg}}\mathcal{V}_{\pi})/\mathcal{I}_{\pi}$, where $\mathcal{I}_{\pi}$ is as in \eqref{balancing-ideal},
\[\mathcal{I}_{\pi}:=\mathrm{span}\left\{xb \otimes v - x\otimes \pi(b)(v):x \in \mathcal{X}, v \in \mathcal{V}_{\pi}, b \in \mathcal{B}\right\}.\]
A straightforward computation shows that $\mathcal{I}_{\pi}$ is a subset of the radical $\mathcal{N}_{\pi}$ defined in (\ref{smooth-radical}). 

\begin{corollary} \label{bigtheta-apply}
Let $X$ be a Hilbert $C^{*}$-module over a $C^{*}$-algebra $B$, $\pi:B\to \mathbb{B}(V_{\pi})$ a $*$-representation. Suppose that $\mathcal{B}\subset B$ is a dense $*$-subalgebra, and $\mathcal{X}\subset X$ and $\mathcal{V}_{\pi}\subset V_{\pi}$ are dense subspaces satisfying $\mathcal{X}\cdot \mathcal{B}\subset \mathcal{X}$ and $\pi(\mathcal{B})\cdot \mathcal{V}_{\pi}\subset \mathcal{V}_{\pi}$. Then the internal tensor product $X\otimes_{B}V_{\pi}$ is equal to the Hilbert space completion of the quotient of $\mathcal{X}\otimes_{\mathcal{B}}^{\mathrm{alg}}\mathcal{V}_{\pi}$ by the subspace $\mathcal{N}_{\pi}/\mathcal{I}_{\pi}$, where $\mathcal{N}_{\pi}$ is as in (\ref{smooth-radical}).
\end{corollary}
\begin{proof}

The inner product \eqref{localisation-inner-product} descends to $\mathcal{X}\otimes_{\mathcal{B}}\mathcal{V}_{\pi}$ and its radical there equals $\mathcal{N}_{\pi}/\mathcal{I}_{\pi}$. Thus the map
\[\mathcal{X}\otimes^{\mathrm{alg}}_{\mathcal{B}}\mathcal{V}_{\pi}\to \mathcal{X}\otimes^{\mathrm{alg}}\mathcal{V}_{\pi}/\mathcal{N}_{\pi},\quad x\otimes v\mapsto [x\otimes v]_{\mathcal{N}_{\pi}},\]
is well-defined and surjective with kernel $\mathcal{N}_{\pi}/\mathcal{I}_{\pi}$. The result now follows from Lemma \ref{lem: dense subspace}.
\end{proof}
\subsection{Induction of representations}
\label{module-integrated-form} 
Let $A$ be a $C^*$-algebra and assume that there is a homomorphism $\alpha: A \to \End_B ^*(X)$ where $X$ is a Hilbert module over $B$ as above. In such a situation, $X$ is called a {\em $C^*$-correspondence} for $(A,B)$, or simply an {\em $(A,B)$-correspondence}. Using the $*$-homomorphism $\alpha$ and $X$, we can induce representations of $B$ to $A$ via the internal tensor product with a representation $\pi:B\to \mathbb{B}(V_{\pi})$. The action 
$$a(x \otimes v) := \alpha(a)(x) \otimes v$$ 
of $A$ on the space $X\otimes V_{\pi}$ gives rise to representation of $A$ on the Hilbert space $X \otimes_B V_{\pi}$ which we will denote $\textnormal{\small Ind}_{B}^{A}(X, \pi),$
and refer to as the $A$-representation {\em induced from $\pi$ via $X$}.

This induction procedure gives us a functor $$\textnormal{\small Ind}_{B}^{A}(X):{\rm Rep}(B)\to {\rm Rep}(A),\quad\pi\mapsto \textnormal{\small Ind}_{B}^{A}(X, \pi)$$ (see e.g. \cite[Proposition 2.69]{Raeburn-Williams}) from the category ${\rm Rep}(B)$ of non-degenerate representations of $B$ with bounded intertwining operators to the corresponding category ${\rm Rep}(A)$ of $A$. It respects unitary equivalence and direct sums. Moreover, if one equips both categories with the Fell topology (which is based on the notion of weak containment), the above defined induction functor is continuous.

Let $G,H$ be locally compact groups and let $X$ be a Hilbert module over $C^*(H)$. Assume that there is a strongly continuous group homomorphism $\alpha: G \to \mathcal{U}(X)$. Just like in the case of a unitary representation on a Hilbert space, we can integrate to a homomorphism $\alpha:C^{*}(G)\to \End_{C^*(H)} ^*(X)$ (see e.g. \cite{Raeburn-Williams}, Proposition C.17]). Now, the induction construction above gives us a functor 
$${\rm Ind}_{C^*(H)}^{C^*(G)}(X) : {\rm Rep}(C^*(H)) \longrightarrow {\rm Rep}(C^*(G)).$$
 
Recalling the categorical equivalence between the unitary representations of a group and representations of its $C^*$-algebra that we discussed above, we can view $X$ as giving a functor 
\begin{equation}\label{induction-group} {\rm Ind}_H^G(X) : {\rm URep}(H) \longrightarrow {\rm URep}(G),
\end{equation}
for the categories of unitary representations.

In the case where one has a Hilbert module $X$ over $C_r^*(H)$ and a strongly continuous group homomorphism $\alpha: G \to \mathcal{U}(\End_{C^*_r(H)} ^*(X))$, we similarly can view $X$ as giving a functor 
\begin{equation}\label{induction-group-temp} {\rm Ind}_H^G(X) : {\rm TRep}(H) \longrightarrow {\rm URep}(G).
\end{equation}

Our main results (see \ref{main-result-one}, \ref{main-result-two}) will be that local theta correspondence actually arises from functors like these in favourable cases. 
\subsection{Morita equivalence}
Two $C^*$-algebras are Morita equivalent if their categories of representations are equivalent. This notion can be made precise in the language of $C^{*}$-correspondences.
\begin{definition}Let $A,B$ be two $C^*$-algebras. We say that $A$ is {\em Morita equivalent} to $B$ if there exists a $C^{*}$-correspondence $X$ for $(A,B)$ such that $X$ is full Hilbert $C^{*}$-module over $B$ and $\alpha:A \xrightarrow{\sim} \mathbb{K}(X)$ is an isomorphism. 
\end{definition}
It is well-known that Morita equivalence defines an equivalence relation. The module $X$ above implements this equivalence via the internal tensor product and is called an $(A,B)$-{\em equivalence bimodule}. In other words the induction functor ${\rm Ind}_{B}^{A}(X)$ is invertible, so that we obtain an equivalence of categories of representations of $A$ and $B$.

\section{Theta correspondence}
\label{sec:thetta}
In this section, we review the theory of theta correspondence, taking the opportunity to set notations. We follow Howe's original account in \cite{Howe-79} closely. 

\subsection{The oscillator representation}
Assume in this subsection that $F$ is a local field of characteristic $0$. Let $W$ be a symplectic vector space over $F$ with the form $\langle {\cdot},{\cdot}\rangle$. 

The {\em Heisenberg group} $\Heis(W)$ is defined as $W\oplus F$ with the multiplication rule
$$(w,t){\cdot}(w',t') := (w+w', t+t'+\tfrac{1}{2}\langle w,w'\rangle).$$

We fix a non-trivial unitary character $\chi : F \to U(1)$. By the Stone-von Neumann Theorem, there exists, up to unitary equivalence, a unique irreducible unitary representation of $\Heis(W)$ with central character $\chi$. We denote this representation $\rho_\chi$ and call it the {\em Heisenberg representation} (of $\Heis(W)$ attached to $\chi$).  

Let $W=X\oplus Y$ is a decomposition of $W$ into maximal isotropic subspaces, then one can realize $\rho_\chi$ on the Hilbert space $L^2(X)$ via the familiar formulas
\begin{eqnarray}\label{schrodinger-formulas} \rho_\chi((x,0,0))(\phi)(x') &:=& \phi(x+x'),\\
\rho_\chi((0,y,0))(\phi)(x') &:=& \chi(\langle x',y \rangle)\phi(x'),\\
\rho_\chi((0,0,a))(\phi)(x') &:=& \chi(a)\phi(x'),
\end{eqnarray}
for all $x,x' \in X$, $y \in Y$ and $a \in F$. This is the so-called {\em Schr\"odinger model} of the Heisenberg representation.

The space of smooth vectors in this representation is given by the space of Schwartz functions on $X$. This is the space of locally constant functions with compact support in the non-archimedean case, and the usual space of smooth functions of rapid decay in the archimedean case.

The group ${\rm Sp}(W)$ of isometries of the symplectic space $W$ acts on the Heisenberg group $\Heis(W)$ as automorphisms via the rule $g{\cdot}(w,t):=(gw,t)$. Let $\mathcal{M}p(W)$ denote the {\em metaplectic group} of $W$. The metaplectic group fits into an exact sequence
$$1 \to \C^1 \to \mathcal{M}p(W) \to {\rm Sp}(W) \to 1.$$

There is a unique, up to equivalence, unitary representation $\omega_\chi$ of $\mathcal{M}p(W)$ on the Hilbert space of $\rho_\chi$ satisfying the covariance property
\begin{equation} \label{covariance} \omega_\chi(\bar{g})\rho_\chi(h)\omega_\chi(\bar{g}^{-1})= \rho_\chi(g{\cdot}h)
\end{equation}
for all $h \in \Heis(W)$ and $g \in {\rm Sp}(W)$ with lift $\bar{g} \in \mathcal{M}p(W)$. This representation is called the {\em oscillator representation}. 

Note that the smooth vectors of $\omega_\chi$ agree with those of $\rho_\chi$ so that the associated smooth representations also act on the same space (which would be a space of Schwartz functions in the Schr\"odinger model as we mentioned earlier).\newline

Assume now that $F$ is a number field with ring of integers $\mathcal{O}$. For a place $v$ of $F$, let $F_v$ denote the local field obtained by the completion of $F$ at $v$. Let $\mathcal{O}_v$ denote the ring of integers of $F_v$. We let $\A$ denote the ring of adeles of $F$. Let $W$ be a symplectic vector space over $F$. As in the local case, we form the attached Heisenberg group $\Heis(W)$. Fixing a standard symplectic basis $\{ e_1, \hdots, e_n, f_1, \hdots, f_n \}$, we consider the $\mathcal{O}$-lattice 
$$L := \langle e_1, \hdots, e_n, f_1, \hdots, f_n \rangle\subrangle{\mathcal{O}}$$ 
inside $W$. For each finite place $v$, we have the $\mathcal{O}_v$-lattice $L_v=L\otimes_\mathcal{O} \mathcal{O}_v$ inside the $F_v$-vector space  $W_v=W \otimes_F F_v$. Notice that 
$$\Heis^o(W_v) := L_v \oplus \mathcal{O}_v$$
is a compact open subgroup of $\Heis(W_v)$. We form the adelic Heisenberg group as the restricted product
$$\Heis_\A :=  \sideset{}{'}\prod_v (\Heis(W_v) : \Heis^o(W_v))$$
of the local Heisenberg groups $\Heis(W_v)$ with respect to the collection compact open subgroups $\Heis^o(W_v)$. As usual, $\Heis(W)$ embeds 
into $\Heis_\A$ as a discrete subgroup.

Let us fix a global character $\chi : \A \to \C^1$ that is trivial on the subgroup $F$. Let $\chi_v : F_v \to \C^1$ denote the local components of $\chi$. Note that for all but finitely many places $v$, we have that $\chi_v$ is trivial on $\mathcal{O}_v$. At each such place $v$, there is a unique vector 
$\phi_v$ in the representation $\rho_{\chi_v}$ that is fixed by $\rho_{\chi_v}(\Heis^o(W_v))$. This vector is always smooth. 

Let us describe the vectors $\phi_v$ in the Schr\"odinger model. Consider the polarization 
$$W_v=X_v \oplus Y_v = \langle e_1, \hdots, e_n \rangle\subrangle{F_v} \oplus \langle f_1, \hdots, f_n \rangle\subrangle{F_v}$$
and realize $\rho_{\chi_v}$ on $L^2(X_v)$. The set 
\begin{equation} \label{lattice} \mathcal{X}_v=\langle e_1, \hdots, e_n \rangle\subrangle{\mathcal{O}_v}
\end{equation}
is an $\mathcal{O}_v$-lattice in $X_v$ and one can check directly from the formulas (\ref{schrodinger-formulas}) that we can take 
$$\phi_v=\chi_{\mathcal{X}_v},$$
to be the characteristic function of the lattice $\mathcal{X}_v$. Using the vectors $\phi_v$, we take the restricted tensor product of the local Heisenberg representations $\rho_{\chi_v}$ to obtain a unitary representation $\rho_\chi$ of $\Heis_\A$. 

Using the standard symplectic basis that we fixed, we identify ${\rm Sp}(W)$ with the $F$ points of the algebraic group ${\rm Sp}={\rm Sp}_{2n}$. Let $J_v$ be the stabilizer of $\Heis^o(W_v)$ inside ${\rm Sp}(W_v)$. Then $J_v$ is compact open for all but finitely many places $v$ and adelic group 
${\rm Sp}_\A$ is given by the restricted product
 $${\rm Sp}_\A =  \sideset{}{'}\prod_v ({\rm Sp}(W_v) : J_v).$$
Assembling the local actions together, the adelic group ${\rm Sp}_\A$ acts on the adelic Heisenberg group $\Heis_\A$ and following the same arguments as in the local setting, we deduce that there is an global metaplectic group $\mathcal{M}p_\A$ satisfying
$$1 \to \C^1 \to \mathcal{M}p_\A \to {\rm Sp}_\A \to 1.$$
and a unitary representation $\omega_\chi$ on the space of $\rho_\chi$ satisfying the analogous covariance property as in (\ref{covariance}).

The group $\mathcal{M}p_\A$ is not a restricted product, but a quotient of one. For all places away from $2$, the subgroups $J_v$ have a unique lifting to $\mathcal{M}p_v$. Using this lifting, we view $J_v$ as a subgroup of $\mathcal{M}p_v$. It turns out that $\mathcal{M}p_\A$ is a quotient of  $\sideset{}{'}\prod_v (\mathcal{M}p(W_v) : J_v)$
by a central subgroup.

It follows from the covariance property (\ref{covariance}), for each $J_v \leq \mathcal{M}p(W_v)$, the vector $\phi_v$ above is fixed under $\omega_{\chi_v}(J_v)$. Using these vectors, we can take the restricted tensor product $\sideset{}{'}\prod_{\phi_v} \omega_{\chi_v}$ of the local oscillator representations to obtain a representation of $\sideset{}{'}\prod_v (\mathcal{M}p(W_v) : J_v)$. The pullback of the oscillator representation $\omega_\chi$ of $\mathcal{M}p_\A$ back to $\sideset{}{'}\prod_v (\mathcal{M}p(W_v) : J_v)$ via the quotient map agrees with 
$\sideset{}{'}\prod_{\phi_v} \omega_{\chi_v}$. 

For later use, we summarize the salient features of the above discussion in the proposition below.

\begin{proposition}
\label{findingfixed}
The smooth oscillator representation $\omega$ of $\mathcal{M}p(W_v)$ admits a fixed vector for $J_v$ in all but finitely many places. In the Schr\"odinger model, the fixed vector is given by the characteristic function of the lattice $\mathcal{X}_v$ given in (\ref{lattice}).
\end{proposition}

\subsection{Dual groups} \label{dual groups} Let $G'$ and $G$ be algebraic reductive subgroups of the algebraic group ${\rm Sp}$  associated to the symplectic space $W$ over a field $F$ such that $G'$ and $G$ are each others' centralizers in ${\rm Sp}$. Such pairs $(G',G)$ are called {\em dual pairs}. These are classified by Howe; a table describing the classification can be conveniently found in \cite[Table 1, p.552]{Sakellaridis-17}. We will call $G$ the \emph{smaller group} if the (Langlands) dual group of $G$ is smaller than $G'$, following \cite[p.552]{Sakellaridis-17}.

Let $(G',G)$ be a dual pair over a number field $F$. Let $v$ be a place of $F$. Since $G'_v=G'(F_v)$ and $G_v=G(F_v)$ commute with each other inside ${\rm Sp}(W_v)$, there is a natural homomorphism $G'_v \times G_v \rightarrow {\rm Sp}(W_v)$. For all classes of dual pairs but one (see Section \ref{exceptional}), Kudla has constructed explicit splittings $\iota_v$
$$\xymatrix{ & \mathcal{M}p(W_v) \ar[d] \\ 
G'_v {\times} G_v \ar@{.>}[ur]^{\iota_v} \ar[r] & {\rm Sp}(W_v)}
$$
We pull-back the local oscillator representation $\omega_{\chi_v}$ of $\mathcal{M}p(W_v)$ that we discussed earlier to $G'_v {\times} G_v$ via these splittings. 

For finite places $v$, we fix the compact open subgroups $K'_v:=G'(\mathcal{O}_v)$ and $K_v=G(\mathcal{O}_v)$. 
For all but finitely many $v$, we have that 
\begin{equation} 
\label{compact-inclusion} 
\iota_v(K'_v) \subseteq J_v, \qquad  \iota_v(K_v) \subseteq J_v,
\end{equation} 
where $J_v$ is the compact open subgroup of ${\rm Sp}(W_v)$ viewed as a compact open subgroup of $\mathcal{M}p(W_v)$ thanks to the canonical splitting mentioned above. The inclusions \eqref{compact-inclusion} allow for a global splitting 
\begin{equation} \label{global-splitting} G'(\A) \times G(\A) \longrightarrow \sideset{}{'}\prod_v (\mathcal{M}p(W_v) : J_v) \to \mathcal{M}p_\A
\end{equation}
via which we obtain the global oscillator representation $\omega_\chi$ of $G'(\A) \times G(\A)$.

\subsubsection{The exceptional case} \label{exceptional} We mentioned above that Kudla's splittings work for all but one class of dual pairs; in this exceptional case, the dual pair is of the form $(G',G)=({\rm Sp}(V'),O(V))$ with $V'$ a symplectic space and $V$ a quadratic space of {\em odd} dimension over $F$. One still can split the orthogonal groups $O(V_v)$, however, the symplectic groups ${\rm Sp}(V'_v)$ admit no splitting. So for all places $v$, we instead set 
$$G'_v:={\rm Mp}(V'_v),$$
the unique double cover of ${\rm Sp}(V'_v)$. This enlargement allows for a splitting and the local oscillator representation $\omega_{\chi_v}$ of $\mathcal{M}p(W_v)$ now pulls back to a representation of $G'_v \times G_v$.

Just as before, the standard compact subgroup $K'_v$ of  ${\rm Sp}(V'_v)$ admits a unique splitting into ${\rm Mp}(V'_v)$ for all but finitely many $v$. We define the global group $G'(\A) $ as a quotient of $\sideset{}{'}\prod_v ({\rm Mp}(V'_v) : K'_v)$ by a certain central subgroup. 
Then $G'(\A)$ is a double cover of ${\rm Sp}(V')(\A)$ and there is a global splitting of $G'(\A) \times G(\A)$ just like in (\ref{global-splitting}).

\subsection{Theta correspondence} 
\label{loca-global-compatibility-2}  
Let $F$ be local field. Let $(G',G)$ be a dual pair over $F$ and let 
$\omega$ be the oscillator representation of $G'\times G$, associated to a non-trivial character of $F$, as discussed above. The local theta correspondence, or local Howe duality, is the fact that the rule
\begin{equation} 
\label{recipe} 
\sigma \leftrightarrow \pi \qquad \textrm{iff} \qquad  {\rm Hom}_{G'\times G}(\omega^\infty, \sigma\otimes \pi) \not= 0
\end{equation}
sets up a bijection between the subsets 
\[\{ \sigma \in {\rm Irr}(G') \mid {\rm Hom}_{G'}(\omega^\infty, \sigma) \not= 0 \}\quad\textnormal{ and}\quad  \{ \pi \in {\rm Irr}(G)  \mid {\rm Hom}_G(\omega^\infty, \pi) \not= 0 \},\] 
of the admissible duals of $G'$ and $G$. This was originally conjectured by Howe in \cite{Howe-79} who also proved it in the archimedean case. In the non-archimedean set-up, it was proven by Waldspurger \cite{Waldspurger-90} when the residue characteristic $p$ of $F$ is not equal to 2. Remaining cases were completed much later by Gan and Takeda \cite{Gan-Takeda-16} and by Gan and Sun \cite{Gan-Sun-17}. If two representations correspond to each other under this bijection, we say that {\em they are each other's theta lifts}.

For the global version, we fix a number field $F$. Let $(G',G)$ be a dual pair over $F$ and let 
$\omega$ be the global oscillator representation of $G'(\A) \times G(\A)$, associated to  a character $\chi : \A \to \C^1$ that is trivial on the subgroup $F$, as discussed above. The global theta correspondence is defined using the same rule (\ref{recipe}) on the admissible duals of the global groups $G'(\A)$, $G(\A)$ and the global oscillator representation $\omega$.

There is a compatibility between the local correspondences and the global one as we describe now. Let $\pi$ be an irreducible admissible representation of $G(\A)$. 
Let $\sideset{}{'}\bigotimes \pi_v$ be the factorization of $\pi$. Assume that $\pi$ enter the global theta correspondence and with corresponding irreducible admissible representation $\sigma$ of $G(\A)$. The local representations $(\pi_v,\mathcal{H}_v)$ are \emph{unramified} (that is $\mathcal{H}_v^{K_v}$ is one dimensional) for all but finitely many places. Let $\sigma_v$ denote the irreducible admissible representation of $G_v$ corresponding to $\pi_v$. If $\pi_v$ is unramified, then $\sigma_v$ is known to be unramified as well. Thus we can take the restricted products of the $\sigma_v$'s. The basic result is that the restricted product $\sideset{}{'}\bigotimes \sigma_v$ gives the factorization of $\sigma$.

\subsection{Type I dual pairs} 
\label{dual-pairs}
We will describe the irreducible reductive dual pairs $(G',G)$ of Type I. Since we will only consider irreducible pairs, we will omit this adjective in the rest of the paper for convenience. 

Let $D$ be a division algebra over $F$ with involution $\iota$. Let $V'$ and $V$ be two $D$-modules equipped with non-degenerate right $\iota$-sesquilinear forms $({\cdot}, {\cdot})'$ and $({\cdot}, {\cdot})$. Assume that one of the forms is $\iota$-hermitian and the other is $\iota$-skew-hermitian. We consider vector space 
$$W:={\rm Hom}_D(V',V).$$ 
Given $w \in W$, we define $w^* \in {\rm Hom}_D(V,V')$ by the rule 
\begin{equation} \label{adjoint}(w(v'),v)=(v', w^*(v))'.
\end{equation}
Then the formula
$$\langle w',w \rangle = {\rm tr}_{D/F} \left (w'w^* \right )$$ 
turns $W$ into a symplectic vector space. As before, we put ${\rm Sp}(W)$ for the isometry group of $W$. We let ${\rm Mp}(W)$ denote the {\em metaplectic group}, the unique non-trivial central double cover of ${\rm Sp}(W)$. We set 
$$G'=G'(V')=\begin{cases} \textrm{the metaplectic group ${\rm Mp}(V')$ if $V'$ is symplectic and {\rm dim}(V) is odd,} \\ 
\textrm{the isometry group of $V'$, otherwise.}
\end{cases}
$$
We define $G=G(V)$ similarly by switching the roles of $V'$ and $V$. The group $G'$ and $G$ act on $W$ via post-multiplication and pre-multiplication by the inverse, respectively. Via these actions, they embed in ${\rm Sp}(W)$. 

\subsection{Stable range} \label{stable-range-def} We say that the dual pair $(G',G)$ is in the {\em stable range}, with $G$ the smaller member, if $V'$ has an isotropic subspace whose dimension is greater than or equal to that of $V$. For later use, we record here the following, see \cite[Corollary 3.3]{Li-89}.

\begin{lemma}
\label{L1-exception} 
Let $(G',G)$ be a Type I dual pair over a local field in the stable range with $G$ the smaller member. Let $\omega$ be the oscillator representation associated to the pair $(G',G)$. If $(G',G)$ is not the ortho-symplectic pair $(O_{2n,2n},Sp_{2n})$ where $O_{2n,2n}$ is the orthogonal group of the split quadratic form in $4n$ variables, then $\omega_{| G}$ has matrix coefficients in $L^{1}(G)$.
\end{lemma}

\subsection{A construction of Jian-Shu Li}
We now explain an explicit construction due to J.-S. Li which lies at the heart of our approach to the local theta correspondence.

Let $(G',G)$ be a Type I dual pair over a local field. Let $\mathbb{S}$ denote the subspace of smooth vectors for the oscillator representation $\omega$ of $G'\times G$. Given an irreducible unitary representation $\pi$ of $G$, one formally introduces a right sesquilinear form $({\cdot}, {\cdot})_\pi$ on $\mathbb{S} \otimes V_{\pi}^\infty$ as follows: if $\Phi = x \otimes y$ and $\Psi = v\otimes w$, then 
\begin{align} \label{Li-form} (\Phi, \Psi)_\pi &:= \int_G \big \langle \Phi,  (\omega\otimes\pi)(s)\Psi \big \rangle \mathrm{d}s \\
\label{Li-form-2} &= \int_G \langle x,\omega(s)(y) \rangle \langle  v,\pi(s)(w)\rangle \mathrm{d}s
\end{align}
Assuming that the integral \eqref{Li-form-2} is convergent (this will be the case if $\omega_{|G}$ is integrable, for example), it is easy to see that this form is Hermitian and $G'$-invariant with respect to the natural action $\omega \otimes {\bf 1}$ of $G'$.   Let $\mathcal{N}_{\pi}$ denote the radical of $({\cdot},{\cdot})_\pi$. Then $\mathcal{N}_{\pi}$ is stabilised by $G'$ and thus the quotient space  
$$\left ( \mathbb{S} \otimes V_{\pi}^\infty \right )/\mathcal{N}_{\pi}$$
affords a $G'$-representation that we will denote by $L(\pi)$. If furthermore $(\ref{Li-form-2})$ is non-negative, then $L(\pi)$ is, after completion, a unitary $G'$-representation.

\begin{theorem}[Li \cite{Li-89}]
Assume that the Hermitian form $({\cdot}, {\cdot})_\pi$ defined in (\ref{Li-form-2}) is convergent and non-negative. If the unitary $G'$-representation $L(\pi)$ is non-zero and irreducible, then $L(\pi)$ is isomorphic to the theta lift $\theta(\pi^*)$ of the contragradient $\pi^*$ of $\pi$.  
\end{theorem}

The theorem is proved in \cite[Section 6]{Li-89}. See also \cite[Section 7]{He-00} and \cite[Section 16.5]{Gan-Ichino-14}. We will consider two settings in which Li's construction is known to work.

\subsection{Li's method for the stable range} \label{Li-results}
For dual pairs in the stable range, the local theta correspondence preserves unitarity as proven by Li (\cite[Theorem A]{Li-89}). This is where Li's method originated. See also \cite{Soudry-89}.

\begin{theorem}\label{Li-unitary}(Li \cite{Li-89})  Let $(G',G)$ be a Type I stable range dual pair over a local field with $G$ the smaller member. Assume that $(G',G)$ is not the pair $(O_{2n,2n},Sp_{2n})$. Let $\pi$ be an irreducible unitary representation of $G$. If $G$ is the metaplectic group, then we assume that $\pi$ does not factor through the symplectic group\footnote{In this case, if $\pi$ factors through the symplectic group, it is easy to see that $\Theta(\pi^\infty)$ is zero.}. 

We have
\begin{enumerate} 
\item the form $({\cdot}, {\cdot})_\pi$ is convergent and non-negative,
\item $L(\pi)$ is nonzero and irreducible. Hence $L(\pi) \simeq \theta(\pi^*)$.
\end{enumerate}
In particular, all of the unitary dual of $G$ enters the theta correspondence and unitary is preserved under the theta correspondence.
\end{theorem}

\subsection{Li's method for tempered representations} Various researchers, notably Przebinda \cite{Przebinda-91, Przebinda-93} and He \cite{He-00, He-03}, see also \cite{Gan-Ichino-14, barbasch_etal_23} for more recent treatments, have worked on extending the scope of Li's method beyond the stable range. We will present one such result which is certainly well-known to the experts in the field.

\begin{theorem} \label{tempered-reps} Let $(G',G)$ be an irreducible Type I dual pair over a local field with $G$ smaller group. Let $\pi$ be a tempered irreducible representation of $G$. Then 
\begin{enumerate}
\item the form $({\cdot}, {\cdot})_\pi$ is convergent and non-negative,
\item $L(\pi)$ is non-zero iff $\pi$ enters the theta correspondence,
\item if $L(\pi)$ is nonzero, then it is irreducible. Hence $L(\pi) \simeq \theta(\pi^*)$.
\end{enumerate} 
In particular, the tempered representations of $G$ entering the theta correspondence are sent to unitary representations.
\end{theorem}

\begin{proof} It follows from Li's analysis of growth of matrix coefficients of the oscillator representation $\omega$ that the matrix coefficients $\omega_{\mid G}$ lie in the Harish-Chandra Schwartz algebra of $G$. As such, they can be integrated against matrix coefficients of tempered representations. Thus the integral in (\ref{Li-form-2}) converges for any tempered irreducible representation $\pi$ of the smaller group $G$. Non-negativity is more subtle but follows from \cite[Theorem A.5]{Harris-Li-Sun}. Irreducibility of $L(\pi)$ follows from \cite[Theorem 1.1]{He-00}. Indeed, tempered representations of the small group are in the range of \emph{semistability} in the terminology of \cite{He-00}. The claim that $L(\pi)$ is non-zero iff $\pi$ enters the theta correspondence follows essentially from the dichotomy principle. One can find an argument therefore for unitary duals in \cite[Theorem B.41]{Harris-Li-Sun} but the argument works generally. See also \cite[Proposition 16.1]{Gan-Ichino-14}.
\end{proof}


\section{The local oscillator bimodule}
\label{sec:localda}

\subsection{General dual pairs}
Let $(G',G)$ be a Type I dual pair over a local field with $G$ the smaller member. We consider the smooth oscillator representation $\omega$ of $G'\times G$ realized on the space of smooth vectors that we will denote by $\s$. Let $\mathcal{S}(G)$ denote the Harish-Chandra Schwartz algebra of $G$. Recall from the proof of Theorem \ref{tempered-reps} that the matrix coefficients of $\omega$ when viewed as a $G$-representation lie in $\mathcal{S}(G)$. In particular, $\omega_{\mid G}$ is a tempered $G$-representation. 

We equip $\s$ with a {\em right} $\mathcal{S}(G)$-module structure as follows:  for  $x \in \s$
\begin{equation} 
\label{eqright-module} 
x {\cdot} b := \int_G b(s)\omega(s^{-1})(x)  \ \mathrm{d}s, \qquad b \in \mathcal{S}(G).
\end{equation}
Note that $x {\cdot} b$ is well-defined and belongs to $\s$ since $\omega_{\mid G}$ is tempered. Next, we equip $\s$ with an $\mathcal{S}(G)$-valued right linear form
\begin{equation} 
\label{right-inner-product} 
\langle x,y \rangle_G (s) := \langle x, \omega(s)(y) \rangle, \qquad x,y \in \s, \ s \in G
\end{equation}
It is easy to check that this form is Hermitian and compatible with the right $\mathcal{S}(G)$-module structure given above: \[\langle x,y \rangle_{G}^{*}=\langle y, x \rangle\subrangle{G},\quad \langle x, y {\cdot} b  \rangle\subrangle{G} = \langle x,y \rangle\subrangle{G} b,\quad x,y\in\s, \ b\in\mathcal{S}(G).\] Recall that $\mathcal{S}(G)$ is a dense subalgebra of the reduced group $C^*$-algebra $C^*_r(G)$ of $G$. 

\begin{proposition} 
\label{tempered-hilbert-module} 
When equipped with the right module structure \eqref{eqright-module} and the form \eqref{right-inner-product}, the space $\s$ becomes a nondegenerate right pre-Hilbert module over $C^*_r(G)$.
\end{proposition}

\begin{proof} 
We just need to prove that the form $\langle {\cdot}, {\cdot} \rangle\subrangle{G}$ is positive definite, that is, for any $x \in \s$, we have 
$\langle x,x \rangle\subrangle{G} \geq 0$ as an element of the $C^*$-algebra $C^{*}_{r}(G)$ and that $\langle x,x \rangle\subrangle{G} =0$ only when $x=0$. 

To show the former, it is enough to show that given an injective representation $\Pi$ of $C^{*}_{r}(G)$, the operator $\Pi(\langle \varphi, \varphi \rangle\subrangle{G})$ is positive for every $\varphi \in \s$. If we prove that $\pi(\langle \varphi, \varphi \rangle\subrangle{G})$ is positive for every $\pi$ in the spectrum of $C^{*}_{r}(G)$, then we will be done by considering representation 
$$\Pi_{r} = \bigoplus_{\pi \in \widehat{C^*_{r}(G)}} \pi$$
which is injective. Therefore it suffices to prove that
$$\pi(\langle x,x \rangle\subrangle{G}) \geq 0$$
as an operator on $V_\pi$ for every $\pi$ in the spectrum of $C^{*}_{r}(G)$. $$\langle v,\pi(\langle x,x \rangle\subrangle{G})v \rangle =  \int_{G} \langle x, \omega(s)x \rangle \langle v,  \pi(s)v \rangle \mathrm{d}s.$$

Let $x, x' \in \s$ and consider the operator $\pi(\langle x,x' \rangle\subrangle{G})$ on $V_\pi$. This operator is determined by the bilinear form 
$$\bigl \langle v, \pi(\langle x,x' \rangle\subrangle{G})(v') \bigr \rangle \geq 0$$
for $v,v' \in V_\pi$. We unfold the left hand side
\begin{align}\nonumber\left \langle v, \pi(\langle x, x' \rangle\subrangle{G})(v') \right \rangle &= 
\left \langle v, \int_{G} \langle x, x' \rangle\subrangle{G}(s) \pi(s)(v') \mathrm{d}s \right \rangle \\
\nonumber&=  \int_{G} \langle x, x' \rangle\subrangle{G}(h) \langle v, \pi(s)(v') \rangle \mathrm{d}s \\
\nonumber&=  \int_{G} \langle x, \omega(h)(x') \rangle \langle v, \pi(s)(v') \rangle \mathrm{d}s\\
\label{eq:innerprodequality} &=(x{\otimes}v, x' {\otimes}v')_\pi
\end{align}
where $({\cdot},{\cdot})_\pi$ is the Hermitian form on $\omega \otimes \pi$ (see \ref{Li-form}). The latter is non-negative by Theorem \ref{tempered-reps} (more precisely, thanks to  \cite[Theorem A.5]{Harris-Li-Sun}). Therefore, we conclude that for $x\in\s$ and $v\in V_{\pi}$ we have
\[\left\langle v, \pi(\langle x,x\rangle_{G})v\right\rangle= (x{\otimes}v, x {\otimes}v)_\pi\geq 0,
\]
which implies that $\langle x,x\rangle_{G}\geq 0$ in $C^{*}_{r}(G)$. Now suppose $\langle x,x\rangle_{G}=0$, so that 
\[\langle \omega(s)x,x\rangle=0,\quad \forall s\in G.\]
Then in particular, for $s=e$ we find that $\langle x,x\rangle=0$, so that $x=0$ in the oscillator representation. Since $\s$ injects into the oscillator representation, we conclude that $x=0$ in $\s$.\end{proof}

\begin{definition} 
\label{oscillator-bimodule-tempered} 
We denote by $\T^r$ the right Hilbert $C_r^*(G)$-module obtained by completing 
the pre-Hilbert module $\s$ over the dense subalgebra $\mathcal{S}(G)$.
\end{definition} 

The superscript $r$ on $\T^r$ stands for reduced. 
Recall the action of $G'$ on the vector space $\s$ via the oscillator representation.

\begin{proposition} \label{prop-left-action} The oscillator representation of $G'$ on $\s \subset \T^r$ extends to a representation $G'\to \mathcal{U}(\T^r)$ of $G'$ by unitary module operators. In particular it induces a $*$-homomorphism 
$$\Omega^r :C^{*}(G')\to \End_{C^*_r(G)} ^*(\T^r).$$ 
\end{proposition}
\begin{proof} 
Since the action of $G'$ and $G$ on $\s$ are unitary and commute with one another, $G'$ preserves the $C_r^*(G)$-valued inner product on $\s$.
By standard $C^{*}$-module techniques, it follows that $\|\omega(g)\|_{\End_{C^*_r(G)} ^*(\T^r)}=1$, so each $\omega(g)$ extends to a unitary module operator $\Omega^r(g)\in\End_{C^*_r(G)}^*(\T^r)$.
 
As discussed in Section \ref{module-integrated-form}, for $a\in L^{1}(G')$ the integral
\begin{equation}\label{left-action} \Omega^r(a):= \int_{G'} a(g) \omega(g)\mathrm{d}g
\end{equation}
converges strongly to an operator in $\End_{C^*_r(G)} ^*(\T^r)$, and defines a $*$-homomorphism $\Omega^r :C^{*}(G')\to \End_{C^*_r(G)} ^*(\T^r).$
\end{proof}

\begin{theorem}
\label{main-result-one} 
Let $(G',G)$ be a Type I  dual pair with $G$ the smaller member. The induction functor ${\rm Ind}_{G}^{G'}(\T^r)$ implemented by the oscillator bimodule $\T^{r}$ captures the theta lifting of tempered irreducible representations of $G$. More precisely, if $\pi$ is an irreducible tempered representation of $G$ then we have an isomorphism
$$\theta(\pi^*) \simeq {\rm Ind}_{G}^{G'}(\T^r, \pi)$$
of $G'$-representations. 
\end{theorem}

\begin{proof}
By Li's construction, $\theta(\pi^{*})$ is realised on the Hilbert space completion of $\mathbb{S}\otimes^{\textnormal{alg}}V^{\infty}_{\pi}/\mathcal{N}_{\pi}$, using the bilinear form $(\cdot,\cdot)_{\pi}$. By Equation \eqref{eq:innerprodequality} we have that \[(\phi\otimes v,\psi\otimes w)_{\pi}=\langle v, \pi(\langle\phi,\psi\rangle_{G})w\rangle_{V_{\pi}},\] so that by Lemma \ref{lem: dense subspace}, the Hilbert space closure of $\mathbb{S}\otimes^{\textnormal{alg}}V^{\infty}_{\pi}/\mathcal{N}_{\pi}$ coincides with $\Theta\otimes_{C^{*}_{r}(G)}V_{\pi}$. It is immediate that this identification is compatible with the $G'$-representation.
\end{proof}

As an immediate corollary we obtain the following.
\begin{corollary} \label{main-result-one-cor} Let $(G',G)$ be a type I dual pair with $G$ the smaller member. Local theta lifting of tempered irreducible representations of $G$ arises from an induction functor associated to a $(C^*(G'),C_r^*(G))$-correspondence (see \ref{C*-correspondence}). Hence it is functorial and continuous. 
\end{corollary}

\subsection{Stable range}
\label{right-module} 
Let $(G',G)$ be a Type I dual pair over a local field that is in the stable range with $G$ the smaller member.  Assume that $(G',G)$ is not the pair $(O_{2n,2n},Sp_{2n})$. 

We again consider the smooth oscillator representation $\omega$ of $G'\times G$ realized on the space of smooth vectors $\s$. 

By Lemma \ref{L1-exception}
the matrix coefficients of $\omega$ when viewed as a $G$-representation lie in not only in $\mathcal{S}(G)$ but also in $L^1(G)$. Thus, the inner product (\ref{right-inner-product}) takes its values in $\mathcal{S}(G) \cap L^1(G)$. The intersection $\mathcal{S}(G) \cap L^1(G)$ is a $*$-algebra that admits dense $*$-preserving inclusions
\[C^{*}_{r}(G)\leftarrow \mathcal{S}(G) \cap L^1(G)\rightarrow C^{*}(G),\]
into both the full and reduced $C^{*}$-algebras of $G$. Since $\mathbb{S}$ is a right  module over $\mathcal{S}(G)$ via (\ref{right-inner-product}), it also a right module over $\mathcal{S}(G) \cap L^1(G)$. Thus, by Lemma \ref{lem: pre-inner-product}, $\mathbb{S}$ admits a completion $\T$ as a right Hilbert $C^{*}$-module over $C^{*}(G)$. 

\begin{proposition} 
\label{construct-1} 
Equipped with the right module structure (\ref{eqright-module}) and the form (\ref{right-inner-product}), the space $\s$ becomes a nondegenerate right pre-Hilbert module over $C^*(G)$.
\end{proposition}
\begin{proof}
The proof is formally identical to the proof of Proposition \ref{tempered-hilbert-module}, but now in order to prove non-negativity of the inner product, we use the injective unitary $G$-representation
$$\Pi = \bigoplus_{\pi \in \widehat{C^*(G)}} \pi.$$
For any $\pi\in \widehat{C^*(G)}$, the equality $\langle v, \pi(\langle x, x' \rangle\subrangle{G})(v') \rangle=(x\otimes v, x'\otimes v')_{\pi}$ is proved as in \eqref{eq:innerprodequality}, and non-negativity in this case follows from Li's result Theorem \ref{Li-unitary}.
\end{proof}
\begin{definition} \label{oscillator-bimodule-stable} We denote by $\T$ the right Hilbert $C^*(G)$-module obtained by completing 
the pre-Hilbert module $\mathbb{S}$ over the subalgebra $\mathcal{S}(G) \cap L^1(G)$.
\end{definition}

The same arguments as in the proof of Proposition \ref{prop-left-action} give us the analogous result in the stable range case.
\begin{proposition} \label{left-action-stable} The oscillator representation of $G'$ on $\s \subset \T$ extends to a representation $G'\to \mathcal{U}(\T)$ of $G'$ by unitary module operators. In particular it induces a $*$-homomorphism 
$$\Omega :C^{*}(G')\to \End_{C^*(G)} ^*(\T).$$ 
\end{proposition}

The proof of Theorem \ref{main-result-one} now goes through verbatim to obtain our second main result.
\begin{theorem}\label{main-result-two} Let $(G',G)$ be a dual pair in the stable range with $G$ the smaller member. Assume that $(G',G)$ is not the pair $(O_{2n,2n},Sp_{2n})$.  The induction functor ${\rm Ind}_{G}^{G'}(\T)$ implemented by the oscillator bimodule $\T$ captures the theta correspondence. More precisely, if $\pi$ is an irreducible unitary representation of $G$ then we have an isomorphism
$$\theta(\pi^*) \simeq {\rm Ind}_{G}^{G'}(\T, \pi)$$
of $G'$-representations. 
\end{theorem}

As an immediate corollary we derive the following.
\begin{corollary} \label{main-result-two-cor} Let $(G',G)$ be a dual pair in the stable range with $G$ the smaller member. Assume that $(G',G)$ is not the pair $(O_{2n,2n},Sp_{2n})$. Local theta correspondence for $(G',G)$ arises from an induction functor associated to a $(C^*(G'),C^*(G))$-correspondence (see \ref{C*-correspondence}). Hence it is functorial and continuous. 
\end{corollary}
To conclude we note that the bimodules $\T$ and $\T^{r}$ are related via the internal tensor product. More precisely, the $*$-homomorphism $C^{*}(G)\to C^{*}_{r}(G)$ can be used to form the tensor product $\T\otimes_{C^{*}(G)}C^{*}_{r}(G)$. The densely defined map
\begin{equation}
\label{akjnakjdndajn}
\mathbb{S}\otimes_{\mathcal{S}(G)}^{\textnormal{alg}}C^{*}_{r}(G)\to \T^{r},\quad  x\otimes f\mapsto x\cdot f,
\end{equation}
preserves the inner products and has dense range. As such, the map \eqref{akjnakjdndajn} induces an isomorphism of $(C^{*}(G'),C^{*}_{r}(G))$-correspondences $\T\otimes_{C^{*}(G)}C^{*}_{r}(G)\xrightarrow{\sim} \T^{r}$.

\subsection{Irreducibility of big theta} 
\label{bigh-theta}

We pause to make a small observation regarding the so-called \emph{big theta lift} and recent work of Loke and Przebinda \cite{Loke-Przebinda}. Let $(G',G)$ be a Type I dual pair over a local field $F$ with $G$ the smaller member. We assume that $F$ is nonarchimedean. As usual, we let $\omega$ be the smooth oscillator representation of $G'\times G$ realized on the space of smooth vectors $\s$.  Given an irreducible smooth\footnote{Recall that $\pi$ is smooth if every $v \in V_\pi$ has a compact open stabilizer in $G$. It is known that irreducible smooth representations are admissible.} representation $(\pi,V_\pi)$ of $G$ that enters the theta correspondence, put 
$$J_\pi := \bigcap_{\phi \in {\rm Hom}_{G}(\s,V_\pi)} {\rm Ker} \ \phi.$$
The space $J_\pi$ is $G' \times G$-invariant and it is known that we have the isomorphism of $G' \times G$-representations
$$\s / J_\pi \simeq V_{\Theta(\pi)}  \otimes  V_{\pi}$$
for some smooth $G'$-representation $(\Theta(\pi),V_{\Theta(\pi)})$. This $G'$-representation $\Theta(\pi)$ is called the {\em big theta lift} of $\pi$.
It serves as a proxy to the theta lift $\theta(\pi)$ of $\pi$ in that $\theta(\pi)$ is the unique irreducible quotient of $\Theta(\pi)$.

A natural question that has importance in certain applications is when $\Theta(\pi)$ equals $\theta(\pi)$, in other words, when $\Theta(\pi)$ is irreducible. We will show that our results in the previous section, together with a recent result of Loke and Przebinda \cite{Loke-Przebinda}, allow us to say something in this direction.

The convolution algebra $C_c^\infty(G)$ of locally constant, compactly supported functions on $G$ is called the {\em Hecke algebra} of $G$ and is denoted $\mathcal{H}(G)$. Every smooth representation of $G$ is a module over $\mathcal{H}(G)$. The recent preprint \cite{Loke-Przebinda} contains the result below.

\begin{proposition} \label{Loke-Przebinda} Let $(G',G)$ be a Type I dual pair over a nonarchimedean local field with $G$ the smaller group. For an irreducible smooth representation $(\pi,V_\pi)$ of $G$ that enters the theta correspondence, we have
$$\Theta(\pi) \simeq \s  \otimes^{\textnormal{alg}}_{\mathcal{H}(G)} V_\pi^\vee$$
where $\otimes^{\textnormal{alg}}_{\mathcal{H}(G)}$ denotes algebraic balanced tensor product over the Hecke algebra $\mathcal{H}(G)$ and $\pi^\vee$ is the smooth dual of $\pi$.
\end{proposition}

Recall that the algebraic balanced tensor product $\s  \otimes^{\textnormal{alg}}_{\mathcal{H}(G)} V_\pi^\vee$ is the quotient
$$\left (\s  \otimes^{\textnormal{alg}}  V_\pi^\vee \right ) / \mathcal{I}_{\mathcal{H}(G)}$$
where $\mathcal{I}_{\mathcal{H}(G)}$ is the subspace
$$ \mathcal{I}_{\mathcal{H}(G)}=\mathrm{span}\left\{ \phi h \otimes v - \phi \otimes hv \ : \ \phi \in \s, v \in V_\pi^\vee, h \in \mathcal{H}(G) \right\}.
$$

The Hecke algebra $\mathcal{H}(G)$ acts on $L^2(G)$ via convolution and this action leads to an injective $*$-homomorphism 
$$\mathcal{H}(G) \hookrightarrow \mathbb{B}(L^2(G)).$$
The closure of the image of this map in operator norm coincides with the reduced group $C^*$-algebra $C^*_r(G)$. The Hecke algebra $\mathcal{H}(G)$ is a dense subalgebra of both the reduced $C^{*}$-algebra $C^*_r(G)$ and the maximal group $C^*$-algebra $C^*(G)$.

 Let $(G',G)$ be a Type I dual pair over a nonarchimedean local field $F$ with $G$ the smaller member. Let $\pi$ be a tempered irreducible representation of $G$. Our Theorem \ref{main-result-one} gives us that 
$$\theta(\pi) \simeq \T^r \otimes_{C^*_r(G)} V_{\pi*}$$
with the right hand side being the Hilbert space completion of the quotient
\begin{equation} \label{pre-completion} \left (\T^r  \otimes^{\mathrm{alg}}  V_{\pi*} \right ) / \mathcal{I}_{C^*_r(G)}
\end{equation}
where $\mathcal{I}_{C^*_r(G)}$ is the subspace
$$ \mathcal{I}_{C^*_r(G)}=\mathrm{span}\left\{ \phi b \otimes v - \phi \otimes bv \ : \ \phi \in \T^r, v \in V_{\pi*}, b \in C^*_r(G) \right\}.
$$
Recall that $ \mathcal{I}_{C^*_r(G)}$ equals the radical of the form $\langle {\cdot}, {\cdot} \rangle$ on $\T^r \otimes^{\mathrm{alg}}  V_{\pi^*}$ defined in (\ref{localisation-inner-product}), which in our setup equals Li's form (\ref{Li-form}) as shown in the proof of Proposition \ref{tempered-hilbert-module}, and the completion is done with respect to this form.

As the Hecke algebra $\mathcal{H}(G)$ is a subalgebra of $C^*_r(G)$, $\s$ is a subspace of $\T^{r}$ and $V_{\pi}^{\infty}$ is a subspace of $V_{\pi}$, it follows that $\mathcal{I}_{\mathcal{H}(G)}$ lies inside $ \mathcal{I}_{C^*_r(G)}$. Together with the fact that $\s\subset \T^{r}$ and $V_{\pi}^{\infty}\subset V_{\pi}$ are dense subspaces, it immediately follows from Corollary \ref{bigtheta-apply} that $\T^r \otimes_{C^*_r(G)} V_{\pi*}$ equals the completion of
$$ \left ( \s  \otimes^{\textnormal{alg}}_{\mathcal{H}(G)} V_{\pi^*}^\infty \right  ) / \left ( \mathcal{N}_{\pi^*} /  \mathcal{I}_{\mathcal{H}(G)} \right),$$
in the norm induced by the inner product $\langle {\cdot}, {\cdot} \rangle$,
where $\mathcal{N}_{\pi^*}$ is the radical of this form restricted to $\s  \otimes^{\textnormal{alg}}V_{\pi^*}^\infty$.
At this point, the above result of Loke and Przebinda and our Theorem \ref{main-result-one} lead to the following conclusion.

\begin{corollary} \label{bigtheta-tempered} Let $(G',G)$ be a Type I dual pair over a nonarchimedean local field with $G$ the smaller group. Let $\pi$ be a tempered irreducible representation of $G$. Then $\Theta(\pi)$ is irreducible
if and only if the quotient 
$$\mathcal{N}_{\pi^*} \big /  \mathcal{I}_{\mathcal{H}(G)}$$
where $\mathcal{N}_{\pi^*}$ is the radical of Li's form (\ref{Li-form}) on $\s  \otimes^{\textnormal{alg}} V_{\pi^*}^\infty$, is trivial.
\end{corollary}

In the same manner, we obtain the below as a consequence of our Theorem \ref{main-result-two}.
\begin{corollary} \label{bigtheta-stable} Let $(G',G)$ be a stable range Type I dual pair over a nonarchimedean local field with $G$ the smaller group. Assume that $(G',G)$ is not the pair $(O_{2n,2n},Sp_{2n})$. Let $\pi$ be a unitary irreducible representation of $G$. Then $\Theta(\pi)$ is irreducible
if and only if the quotient 
$$\mathcal{N}_{\pi^{*}} \big /  \mathcal{I}_{\mathcal{H}(G)}$$
where $\mathcal{N}_{\pi^{*}}$ is the radical of Li's form (\ref{Li-form}) on $\s  \otimes^{\mathrm{alg}} V_{\pi^*}^\infty$, is trivial.
\end{corollary}

We should mention that there are some results in the literature concerning the irreducibility of big theta. In the non-archimedean case, irreducibility of big theta is known in the (almost) equal rank case by \cite[Proposition 5.4]{Gan-Ichino-14} and for those tempered representations of the big group satisfying certain hypothesis on their $L$-parameters by  \cite[Proposition 5.4]{Atobe-Gan-17}. In the stable range setting, irreducibility of big theta is proven in \cite[Theorem A]{Loke-Ma} in the archimedean case and in \cite[Corollary 7.3]{Chen-Zou} in the non-archimedean case. 


\section{The image of $\Omega$} 
\label{local-Rallis}

In this section, we will prove that for two of the three families of Type I dual pairs, the images of the Hilbert $C^{*}$-module representations 
$$\Omega^r :C^{*}(G')\to \End_{C^*_r(G)} ^*(\T^r)$$
and 
$$\Omega:C^{*}(G')\to \End_{C^*(G)} ^*(\T)$$
of $C^*(G')$ defined earlier in Prop. \ref{prop-left-action} and Prop. \ref{left-action-stable} contain the sub-$C^*$-algebra of compact operators (see Sect. \ref{Hilbert_modules}). Our main tool will be the so-called local Siegel-Weil formula of Gan and Ichino \cite{Gan-Ichino-14}.

The classification of dual pairs by Howe tells us that Type I pairs fall into three groups: ortho-symplectic pairs, unitary pairs and quaternionic pairs. In order to be able to utilize the result of Gan and Ichino \cite{Gan-Ichino-14} in this Section, we will assume their hypotheses that
\begin{itemize}
\item $F$ is non-archimedean, and 
\item $(G',G)$ is either an ortho-symplectic pair or a unitary pair.
\end{itemize} 

This means that we take the division algebra $D$ of Section (\ref{dual-pairs}) to be either $F$ itself or a quadratic extension of $F$. In the former case, one of the groups of the dual pair is an orthogonal group and the other one is either a symplectic group or a metaplectic group depending on the parity of the dimension of the orthogonal space. Such a pair is often called an ortho-symplectic pair. In the latter case, both groups are unitary and hence the pair is called a unitary pair. 

We put 
$$\ell = \begin{cases} m' - m, \qquad \ \ \textrm{when $G',G$ are unitary}, \\
 m'-m-1, \quad \textrm{when $G'$ is orthogonal,} \\
 m'-m+1, \quad \textrm{when $G$ is orthogonal.} 
  \\ 
 \end{cases}$$ 
Here $m'$ and $m$ denote the dimension of $V'$ and $V$ over $D$ respectively. We assume that $\ell$ is always positive, which corresponds to assuming that $G$ is the smaller group (see Section \ref{dual groups}).

We let $\chi_{_V}, \chi_{_{V'}}$ denote the auxilllary splitting characters of $D^\times$ that we fix for the oscillator representation $\omega=\omega_{V',V}$ of $G' \times G$. These are described precisely in \cite[3.2]{Gan-Ichino-14}. 

\subsection{Doubling and the degenerate principal series} 

For the main result of this section, we will need to the basics of the doubling method and the degenerate principal series which we quickly summarize. For complete detail, we refer the reader to \cite{Gan-Ichino-14} which we follow.

Consider the doubled space
$$\V'=V'\oplus (-V')$$ 
where $-V'$ denotes the space $V'$ with the form $-({\cdot},{\cdot})'$. Let $\G'=G'(\V')$ denote the isometry group of $\V'$. Identifying $G'(-V')=G'(V')=G'$, we embed 
$$i: G' \times G' \hookrightarrow \G'.$$ 
We have a complete polarization
$$\V' = V'_d \oplus V'_{ad}$$
with $V'_d=\{ (v',v') : v' \in V'\}$ and $V'_{ad}=\{ (v',-v') : v' \in V'\}$.

The stabilizer $\mathbb{P}$ of  $V'_d$ in $\G'$ is a Siegel parabolic subgroup with Levi subgroup ${\rm GL}(V'_{d})$. Given $s \in \C$ and a character $\psi$ of $D^\times$, we let 
$$I_{\mathbb{P}}^{\G'}(s,\psi) := {\rm Ind}_{\mathbb{P}}^{\G'}(\psi |{\cdot}|^s_D)$$
denote the normalized induced representation of $\G'$. Here we regard $\psi |{\cdot}|^s_D$ as a character of $\mathbb{P}$ via the composition
$$\mathbb{P} \longrightarrow {\rm GL}(V'_{d}) \xrightarrow{{\rm det}} D^\times.$$

As $(G'(\V'),G)$ is a dual pair, we have an oscillator representation 
$\bm{\omega}$ of $G'(\V') \times G$, realized on $\mathcal{S}(V'_{ad} \otimes V)$, which satisfies
\begin{equation}\label{double-identity} \omega \otimes \left ( \overline{\omega} \otimes \chi_V \right ) \simeq \bm{\omega}
\end{equation}
as $G'{\times}G'$-representations. For $\phi \in \mathcal{S}(V'_{ad} \otimes V)$ and $g \in \G'$, put 
$$\mathcal{F}_\phi(g):= \bm{\omega}(g)(\phi)(0).$$
Then the rule $\phi \mapsto \mathcal{F}_\phi$ gives a $\G'$-equivariant map 
$$\mathcal{F} : \bm{\omega} \longrightarrow I_{\mathbb{P}}^{\G'}(-\tfrac{\ell}{2},\chi_{_{V}})$$
whose image we denote by 
$$R(V).$$
A key property of the map $\mathcal{F}$ is that for $\phi_1 \in \omega$ and $\overline{\phi}_2 \in \overline{\omega}$ and $g \in G'$, we have 
$$\mathcal{F}_{\phi_1 \otimes \overline{\phi}_2}(i(g,1))= \langle \omega(g)(\phi_1), \phi_2 \rangle.$$
Here we used the isomorphism (\ref{double-identity}) but suppressed it notationally for convenience.

\subsection{The intertwiner} Let $\mathbb{H}$ be the quadratic space over $F$ of dimension $2$ with quadratic form given by the symmetric matrix 
$\left ( \begin{smallmatrix} 0 & 1 \\ 1 & 0 \end{smallmatrix} \right )$. The space $\mathbb{H}$ is called the hyperbolic plane. We introduce the following enlargement of the space $V$: 
$$V^\dagger := V \oplus \mathbb{H}^{\ell}.$$

Let $\omega_{\dagger}=\omega_{V',V^\dagger}$ denote the oscillator representation of $G' \times G(V^\dagger)$ for which one can use the same auxillary characters $\chi_{_{V'}}, \chi_{_{V}}$. We set $\bm{\omega}_\dagger$ denote the oscillator representation of $\G' \times G(V^\dagger)$ and proceeding as before, we obtain a $\G'$-equivariant map 
$$\mathcal{F} : \bm{\omega}_\dagger \longrightarrow I_{\mathbb{P}}^{\G'}(\tfrac{\ell}{2},\chi_{_{V}})$$
whose image we denote by $R(V^\dagger)$. Notice that the dimension of $V^\dagger$ is greater than that of $V'$. This leads to the following observation which we record for later use.

\begin{lemma} 
\label{L1-dagger} 
The oscillator representation $\omega_\dagger$, viewed as a $G'$-representation, is integrable. 
\end{lemma}

\begin{proof} Noting that the dimension of $V^\dagger$ is $m+2\ell$, it follows directly from Theorem 3.2 of \cite{Li-89} that the matrix coefficients of the $G'$-representation $\omega_\dagger$ are in $L^1(G')$. \end{proof}

Crucial to us is the fact there is a $\G'$-intertwiner 
$$\mathcal{M} : I_{\mathbb{P}}^{\G'}(\tfrac{\ell}{2},\chi_{_{V}}) \longrightarrow I_{\mathbb{P}}^{\G'}(-\tfrac{\ell}{2},\chi_{_{V}})$$ 
which restricts to a surjection
$$\mathcal{M} : R(V^\dagger) \twoheadrightarrow R(V).$$ 

\subsection{A local Siegel-Weil formula} 
Let $\phi_1,\phi_2,\phi_3,\phi_4 \in \omega$. Following \cite[Sections 17 and 18]{Gan-Ichino-14}, let us define 
\begin{align*} \mathcal{E}(\phi_1,\phi_2,\phi_3,\phi_4) &:= \int_G' \mathcal{F}^\dagger_{\phi_1 \otimes \overline{\phi}_2}(i(g,1)) \overline{\mathcal{F}_{\phi_3 \otimes \overline{\phi}_4}(i(g,1))}\mathrm{d}g \\
&= \int_G' \mathcal{F}^\dagger_{\phi_1 \otimes \overline{\phi}_2}(i(g,1)) \overline{\langle \omega(g)(\phi_3), \phi_4 \rangle}\mathrm{d}g
\end{align*}
where $\mathcal{F}^\dagger_{\phi_1 \otimes \overline{\phi}_2} \in R(V^\dagger)$ is a preimage of $ \mathcal{F}_{\phi_1 \otimes \overline{\phi}_2} \in R(V)$ under the intertwiner $\mathcal{M}$ above. Observe that by Lemma \ref{L1-dagger}, the function $\mathcal{F}^\dagger_{\phi_1 \otimes \overline{\phi}_2}(i(g,1))$ belongs to $L^1(G')$ and hence the integral converges. The definition can be shown to be independent of the choice of the preimage. 

We also define 
$$\mathcal{I}(\phi_1,\phi_2,\phi_3,\phi_4):= \int_G \langle \omega(s)(\phi_1), \phi_3 \rangle  \overline{\langle \omega(s)(\phi_2), \phi_4 \rangle} \mathrm{d}s.$$
As $\omega$ is integrable as a $G$-representation, this integral converges. 

In \cite[Theorem 17.2]{Gan-Ichino-14}, we find a proof of the local Siegel-Weil formula whose content is that $\mathcal{E}$ and $\mathcal{I}$ both belong to a certain space of intertwiners which turns out to be one dimensional, and hence they are proportional. We let $\lambda$ be the constant of proportionality (which depends on the chosen Haar measures for $G'$ and $G$) so that we have 
\begin{equation} \label{LRIPF} \mathcal{E}(\phi_1,\phi_2,\phi_3,\phi_4) = \lambda {\cdot }\mathcal{I}(\phi_1,\phi_2,\phi_3,\phi_4)
\end{equation}
for all $\phi_1,\phi_2,\phi_3,\phi_4 \in \omega$.  

\subsection{The image of $\Omega$}
We are ready to prove the main result of this section. 

\begin{theorem} 
\label{left-surject} 
Let $(G',G)$ be a type I dual pair with $G$ the smaller group. Assume that $F$ is a non-archimedean and that $(G',G)$ is either ortho-symplectic or unitary. Then the image of the homomorphism 
$$\Omega^r: C^*(G') \to \End_{C^*_r(G)} ^*(\T^r)$$ 
contains $\mathbb{K}_{C^*(G)}(\T)$. 
 
 Moreover, if $(G',G)$ is in the stable range and $(G',G)$ is not the pair $(O_{2n,2n},Sp_{2n})$, then the image of the homomorphism 
$$\Omega: C^*(G') \to \End_{C^*(G)} ^*(\T)$$ 
again contains $\mathbb{K}_{C^*(G)}(\T)$. 
\end{theorem}

\begin{proof} Let $x,y \in \mathcal{S}(X,\mathcal{F}) \subset \T$ and let $T_{x,y}$ denote the associated rank one $C^*(G)$-module operator on $\T$: 
$$T_{x,y}(z)= x{\cdot}\langle y,z \rangle\subrangle{G}$$
for all $z \in \mathcal{S}(X,\mathcal{F})$. 

Let $\mathcal{F}^\dagger_{y\otimes \overline{x}} \in R(V^\dagger)$ be any preimage under the intertwiner $\mathcal{M}$ of $\mathcal{F}_{y\otimes \overline{x}} \in R(V)$. By Lemma \ref{L1-dagger}, $\mathcal{F}^\dagger_{y\otimes \overline{x}}$ belongs to $L^1(G')$. 
We will show that $\Omega(\mathcal{F}^\dagger_{y\otimes \overline{x}})$ equals a scalar multiple of $T_{x,y}$. This is sufficient since rank one operators are dense in $\mathbb{K}_{C^*(G)}(\T)$ by its definition and since $\mathcal{S}(X,\mathcal{F})$ is dense in $\T$. 

Recall that Hilbert spaces are taken to be right linear in this paper. Let $\lambda$ be the constant of proportionality in (\ref{LRIPF}). Given $z,u \in \mathcal{S}(X,\mathcal{F})$, we have
\begin{align*} \langle u, \lambda {\cdot }T_{x,y}(z) \rangle &= \lambda \langle u, T_{x,y}(z) \rangle \\ 
&= \lambda {\cdot }\langle u, x{\cdot}\langle y,z \rangle\subrangle{G} \rangle \\
&= \lambda {\cdot }\langle u, \int_G \langle y,z \rangle\subrangle{G}(s^{-1}) \omega(s)(x) ds \rangle \\
&=\lambda {\cdot } \int_G \langle y,z \rangle\subrangle{G}(s^{-1}) \langle u, \omega(s)(x) \rangle ds \\
&= \lambda {\cdot }\int_G \langle \omega(s)(y),z \rangle \overline{\langle \omega(s)(x), u \rangle} ds \\
&= \lambda {\cdot } \mathcal{I}(y,x,z,u)
\end{align*}
On the other hand, we have
\begin{align*} \langle u, \Omega(\mathcal{F}^\dagger_{y\otimes \overline{x}})(z) \rangle &= \langle u, \int_{G'} \mathcal{F}^\dagger_{y\otimes \overline{x}}(i(g,1)) \omega(g)(z) dg \rangle \\
&= \int_{G'} \mathcal{F}^\dagger_{y\otimes \overline{x}}(i(g,1)) \langle u, \omega(g)(z)\rangle dg \\
&= \int_{G'} \mathcal{F}^\dagger_{y\otimes \overline{x}}(i(g,1)) \overline{\langle \omega(g)(z), u \rangle} dg \\
&= \mathcal{E}(y,x,z,u)
\end{align*}
We conclude from identity (\ref{LRIPF}) that $\langle u, \lambda {\cdot } T_{x,y}(z) \rangle$ equals $\langle u, \Omega(\mathcal{F}^\dagger_{y\otimes \overline{x}})(z) \rangle$, and hence that $\lambda {\cdot } T_{x,y}$ equals $\Omega(\mathcal{F}^\dagger_{y\otimes \overline{x}})$  as desired.
\end{proof}


\begin{corollary} 
With the above hypotheses, the induction functors ${\rm Ind}_{C_r^*(G)}^{C^*(G')}(\T^r)$  and ${\rm Ind}_{C^*(G)}^{C^*(G')}(\T)$ take irreducible representations to irreducible ones.  
\end{corollary}

\begin{proof}
We take a tempered irreducible representation $(\mathcal{H},\pi)$ of $G$ and write $(\mathcal{H}',\pi'):={\rm Ind}_{C^*_r(G)}^{C^*(G')}(\T,\pi)$. Since $G$ is type $I$, $\pi(C^*_r(G))\supseteq \mathbb{K}(\mathcal{H})$ contains all compact operators. Therefore, $\mathbb{K}(\mathcal{H}')\subseteq \mathrm{im}(\mathbb{K}_{C^*_r(G)}(\T)\to \mathbb{B}(\mathcal{H}'))$ and as an immediate consequence of Theorem \ref{left-surject} we conclude that $\pi'(C^*_r(G'))\supseteq \mathbb{K}(\mathcal{H}')$. Since $\pi'(C^*_r(G'))\supseteq \mathbb{K}(\mathcal{H}')$, $\pi'$ is irreducible. The argument for unitary representations in the stable range case is the same.
\end{proof}

Note that we already knew this since these functors agree with the local theta correspondence on the irreducible unitary representations as proven in Theorems \ref{main-result-one} and \ref{main-result-two}.

Building on Theorem \ref{left-surject}, when $(G',G)$ is a type I dual pair over a non-archimedean field that is either ortho-symplectic or unitary, we form the subquotient
$$C^*_\theta(G'):=(\Omega^r)^{-1}(\mathbb{K}_{C^*_r(G)}(\T^r))/\ker\Omega^r,$$
which is an ideal in $C^*(G')/\ker\Omega^r$. As such, the spectrum $\widehat{C^*_\theta(G')}$ is a locally closed set in $\widehat{G}'$ as it forms an open set in the closed set $\widehat{G}'\setminus \widehat{\ker\Omega^r}$. We also write $C^*_\theta(G)$ for the closed ideal in $C^*_r(G)$ spanned by $\langle \T^r,\T^r\rangle\subrangle{G}$. The spectrum $\widehat{C^*_\theta(G)}$ is an open set in $\hat{G}_t$. By standard arguments, we deduce the following corollary to Theorem \ref{left-surject}.

\begin{corollary}
\label{thetamorita}
Let $(G',G)$ be a type I dual pair with $G$ the smaller group. Assume that $F$ is a non-archimedean and that $(G',G)$ is either ortho-symplectic or unitary. Then the $(C^*(G'),C^*_r(G))$-correspondence $\T^r$ restricts to a Morita equivalence of $C^*_\theta(G')$ and $C^*_\theta(G)$ that implements a homeomorphism from the unitary dual of $C^*_\theta(G)$ to the unitary dual of $C^*_\theta(G')$ via Rieffel induction. In particular, the domain of the theta correspondence in $\hat{G}_t$ is the open set $\widehat{C^*_\theta(G)}$ and its range is the locally closed set $\widehat{C^*_\theta(G')}\subseteq \hat{G}'$
\end{corollary}
An analogue of this corollary also exists for the stable range case of course, we just need to replace $C^*_r(G)$ with $C^*(G)$.  In fact, recall that in the stable rang case $C^*_\theta(G)=C^*(G)$ as all of the unitary dual of the small group enters the theta correspondence.

\begin{remark} In \cite{Kakuhama-22}, Kakuhama pursued analogues of the results of Gan and Ichino in \cite{Gan-Ichino-14} (some of which we used above) in the remaining quaternionic case ($F$ still taken to be non-archimedean). Unfortunately, Kakuhama only considers the case $\ell=1$ when it comes to the construction of the intertwiners $\mathcal{E},\mathcal{I}$. Concerning the archimedean case, it seems plausible that the ingredients coming from \cite{Gan-Ichino-14} have their archimedean analogues but we have not found such results in the literature.
\end{remark}

\section{The global oscillator bimodule}
\label{globalmodule}

The main results of the recent paper \cite{GMS-24} allow us to bundle together the above local $C^*$-correspondences to obtain a $C^*$-correspondence that captures the global theta correspondence. Let $F$ be a number field with its set of places denoted $S(F)$ and $S_{\rm fin}(F)\subseteq S(F)$ the subset of finite places. As above we write $\A:=\sideset{}{'}{\prod}_{v\in S(F)} (F_v : \mathcal{O}_v)$ and we also write 
$$\A_{\rm fin}:=\sideset{}{'}{\prod}_{v\in S_{\rm fin}(F)} (F_v : \mathcal{O}_v).$$
Let $(G',G)$ be a Type I stable range dual pair over $F$ with $G$ the small group. Assume that $(G',G)$ is not the pair $(O_{2n,2n},Sp_{2n})$. Fix a global character $\chi : \A \to \C^1$ that is trivial on the subgroup $F$ and consider the global oscillator representation $\omega_\chi$ of $G'(\A) \times G(\A)$.  

\subsection{Globalizing the stable range case} For every place $v$ of $F$, Theorem \ref{main-result-two} gives us a $\left (C^*(G'_v), C^*(G_v \right))$-correspondence $\T_v$ associated to the dual pair $(G'_v,G_v)$. For finite places $v$, let $p'_v\in C^*(K_v')\subseteq C^*(G'_v)$ and $p_v\in C^*(K_v)\subseteq C^*(G_v)$ be the projections corresponding to the characteristic functions of the compact open subgroups $K'_v$ and $K_v$ of $G'_v$ and $G_v$. These are the support projections in the corresponding $C^*$-algebras of the trivial representation, projecting onto fixed vectors of the respective compact open subgroup. We let  $\phi_v \in \T_v$ denote the vector in Proposition \ref{findingfixed}. It follows from Proposition \ref{findingfixed} and (\ref{compact-inclusion}) that for all but finitely many places $v$, we have
$$\langle \phi_v, \phi_v \rangle\subrangle{G_v} = p_v,$$
where $\langle {\cdot}, {\cdot}\rangle\subrangle{G_v}$ is the $C^*(G_v)$-valued inner product defined in (\ref{right-inner-product}). This means that the collection of pairs $(\T_v, \phi_v)_v$ is compatible with 
$(C^*(G_v),p_v)_v$ in the sense of \cite[Definition 4.1]{GMS-24}. 

It again follows from Proposition \ref{findingfixed} and (\ref{compact-inclusion}) that for all but finitely many places $v$, we have
$$p'_v{\cdot}\phi_v = \phi_v$$
with respect to the left action of $C^*(G'_v)$ on $\T_v$. This means that the left actions of $(C^*(G'_v),p'_v)_v$ on $(\T_v,\phi_v)_v$ are compatible with respect to right module structures over $(C^*(G_v),p_v)_v$ in the sense of \cite[Definition 5.1]{GMS-24}.

It follows from  \cite[Theorem 4.2 and Proposition 5.2]{GMS-24} that we can bundle our local bimodules together to get a global one: 
$$\T_\A:=\sideset{}{'}{\bigotimes}_{v \in S(F)} (\T_v, \phi_v).$$
This is a $C^*$-correspondence for the pair
$$\left ( \sideset{}{'}{\bigotimes}_{v \in S(F)} (C^*(G'_v), p'_v),  \sideset{}{'}{\bigotimes}_{v \in S(F)} (C^*(G_v), p_v)  \right).$$

Let us assume for now that $(G',G)$ is not an exceptional dual pair. Then both of the global groups $G'(\A)$ and $G(\A)$ are restricted products of the local groups and using \cite[Corollary 3.5]{GMS-24}, we conclude that $\T_\A$ is a  $(C^*(G'(\A)), C^*(G(\A)))$-correspondence as claimed. 

Let $(\pi,\mathcal{H}_\pi)$ be a unitary irreducible representation of $G(\A)$ with factorization $\sideset{}{'}{\bigotimes} \pi_v$. By  \cite[Proposition 4.5]{GMS-24}, we have
$$\T_\A \otimes_{C^*(G(\A))} \mathcal{H}_\pi \simeq  \sideset{}{'}{\bigotimes}_{v \in S(F)} \T_v \otimes_{C^*(G_v)} \mathcal{H}_{\pi_v}.
$$
By Theorem \ref{main-result-two}, we know that at each place, the representation  $\omega_{\chi_v} \otimes  {\bf 1}$ afforded on the Hilbert space $ \T_v \otimes_{C^*(G_v)} \mathcal{H}_{\pi_v}$ is the theta lifts of the contragradient of $\pi_v$. It follows now from the local-global compatibility of the theta correspondence (see Subsection \ref{loca-global-compatibility-2}) that the representation  $\omega_\chi \otimes {\bf 1}$ afforded on the Hilbert space $\T_\A \otimes_{C^*(G(\A))} \mathcal{H}_\pi$ is the contragradient of the (global) theta lift of $\pi$ as claimed.

Let us consider the case of an exceptional dual pair $(G',G)$. In this case, one of the global groups is not the restricted product of the local groups but is a quotient of one. This fact requires us to modify the above arguments slightly. Recall that in this case, one group is odd orthogonal group and the other is symplectic. Assume first that the symplectic group is $G$. As explained in Section \ref{exceptional}, the local groups $G_v$ are taken to be ${\rm Mp}(V_v)$ and the global group $G(\A)$ is a {\em quotient} of the restricted product $\mathbb{G}:=\sideset{}{'}{\prod}_{v \in S(F)} (G_v, K_v)$. As such, there is a natural $*$-homomorphism (see e.g. \cite[Proposition 8.C.8]{Bekka-delaHarpe})
\begin{equation} 
\label{quotient} 
C^*(\mathbb{G}) \to C^*(G(\A)).
\end{equation}
It follows from Corollary 3.5 of \cite{GMS-24} that 
$$C^*(\mathbb{G}) \simeq \sideset{}{'}{\bigotimes}_{v \in S(F)} (C^*(G_v), p_v).$$
On the orthogonal side, there are no issues and the global group $G'(\A)$ is just the restricted product of the local orthogonal groups.
Therefore $\T_A$ is a $(C^*(G'(\A)),C^*(\mathbb{G}))$-correspondence. 
To obtain a $(C^*(G'(\A)),C^*(G(\A)))$-correspondence, we consider 
$$\T_A' := \T_\A \otimes_{C^*(\mathbb{G})} C^*(G(\A))$$
where we view $C^*(G(\A))$ as a right Hilbert $C^*$-module over itself in the standard way and utilize the $*$-homomorphism (\ref{quotient}) to turn it into a $(C^*(\mathbb{G}), C^*(G(\A)))$-correspondence.

Given a unitary representation $(\pi,\mathcal{H}_\pi)$ of $G(\A)$, we observe that 
$$\T'_\A \otimes_{C^*(G(\A))} \mathcal{H}_\pi \simeq \T_\A \otimes_{C^*(\mathbb{G})}  C^*(G(\A)) \otimes_{ C^*(G(\A))} \mathcal{H}_\pi \simeq \T_\A \otimes_{C^*(\mathbb{G})} \mathcal{H}_\pi$$
so that we are really pulling back $\pi$ to $C^*(\mathbb{G})$ via the map (\ref{quotient}) and then inducing it to $C^*(G'(\A))$ via $\T_\A$. This is consistent with the local-global compatibility of the theta correspondence in this exceptional dual pair case.

Finally, let us now swap the roles and consider the easier case where the symplectic group is $G'$. In this case $\T_\A$ is a right Hilbert $C^*$-module over $C^*(G(\A))$ and the group $\mathbb{G'}$ acts on the Hilbert $C^*$-module $\T_\A$ unitarily. This unitary action factors through $G'(\A)$ and integrates to a left action of $C^*(G'(\A))$. Thus  $\T_\A$ is a  $(C^*(G'(\A)), C^*(G(\A)))$-correspondence as claimed. Let us summarize our discussion.

\begin{theorem} 
\label{main-result-three}  
Let $F$ be a number field with set of places $S(F)$. Let $(G',G)$ be a Type I stable range dual pair over $F$ with $G$ the small group. Assume that $(G',G)$ is not the pair $(O_{2n,2n},Sp_{2n})$. Fix a global character $\chi : \A \to \C^1$ that is trivial on the subgroup $F$ and consider the global oscillator representation $\omega_\chi$ of $G'(\A) \times G(\A)$. 
The $(C^*(G'(\A)), C^*(G(\A)))$-correspondence
$$\T_\A:=\sideset{}{'}{\bigotimes}_{v \in S(F)} (\T_v, \phi_v)$$
where $\T_v$ is the $(C^*(G'_v),C^*(G_v))$-correspondence associated via Theorem \ref{main-result-two} to the local dual pair $(G'_v,G_v)$ and $\phi_v$ is the vector in Proposition \ref{findingfixed}, satisfies the following: 
if $(\pi,\mathcal{H}_\pi)$ is a unitary irreducible representation of $G(\A)$, then the unitary representation $\omega_\chi \otimes {\bf 1}$ of $G'(\A)$ realized on the Hilbert space $\T_\A \otimes_{C^*(G(\A))} \mathcal{H}_\pi$ is the theta lift of the contragradient of $\pi$.
\end{theorem}

\subsection{Globalizing the general type I case} 
Let now $(G',G)$ be a Type I dual pair over $F$ with $G$ the small group. Everything we have done above can be adapted to this case. Indeed, in this case, Theorem \ref{main-result-one} gives us $\left ( C^*(G'_v),C^*_r(G_v) \right )$-correspondences $\T_v^r$. Our arguments above apply verbatim and we can bundle these local correspondences together to obtain 
$$\T_\A^r:=\sideset{}{'}{\bigotimes}_{v \in S(F)} (\T_v^r, \phi_v).$$
which is a $C^*$-correspondence for the pair
$$\left (C^*(G'(\A)),C^*_r(G(\A))\right)=\left ( \sideset{}{'}{\bigotimes}_{v \in S(F)} (C^*(G'_v), p'_v),  \sideset{}{'}{\bigotimes}_{v \in S(F)} (C^*_r(G_v), p_v)  \right).$$
Arguing as above, we obtain the following.

\begin{theorem} 
\label{main-result-four}  
Let $F$ be a number field with set of places $S(F)$. Let $(G',G)$ be a Type I dual pair over $F$ with $G$ the small group. Fix a global character $\chi : \A \to \C^1$ that is trivial on the subgroup $F$ and consider the global oscillator representation $\omega_\chi$ of $G'(\A) \times G(\A)$. 
The $\left (C^*(G'(\A)), C^*_r(G(\A)) \right )$-correspondence
$$\T_\A^r:=\sideset{}{'}{\bigotimes}_{v \in S(F)} (\T_v, \phi_v)$$
where $\T_v^r$ is the $\left (C^*(G'_v),C^*_r(G_v) \right )$-correspondence associated via Theorem \ref{main-result-one} to the local dual pair $(G'_v,G_v)$ and $\phi_v$ is the vector in Proposition \ref{findingfixed}, satisfies the following: 
if $(\pi,\mathcal{H}_\pi)$ is a tempered irreducible representation of $G(\A)$, then the unitary representation $\omega_\chi \otimes {\bf 1}$ of $G'(\A)$ realized on the Hilbert space $\T_\A^r \otimes_{C^*_r(G(\A))} \mathcal{H}_\pi$ is the theta lift of the contragradient of $\pi$.
\end{theorem}

\subsection{Domain and range of the theta correspondence} In the same spirit as Corollary \ref{thetamorita}, we can describe the precise domain of definition and range for theta correspondence over the finite adeles. We assume that $(G',G)$ be a type I dual pair over the number field $F$ that is either orthosymplectic or unitary. For each $v\in S_{\rm fin}(F)$, the discussion preceding Corollary \ref{thetamorita} produces ideals 
$$C^*_\theta(G(F_v))\subseteq C^*_r(G(F_v))\quad \mbox{and}\quad C^*_\theta(G'(F_v))\subseteq C^*(G'(F_v))/\ker\Omega^{r},$$
and a the $(C^*(G'(F_v)),C^*_r(G(F_v))$-correspondence $\T_v^r$ constructed as above is in fact a Morita equivalence from $C^*_\theta(G(F_v))$ to $C^*_\theta(G'(F_v))$. In the considerations above, we note that $p_v:=\langle\phi_v,\phi_v\rangle_{C^*_r(G(F_v))}\in C^*_\theta (G(F_v))$ and $p_v'$ acts as a $C^*_r(G(F_v))$-compact operator so by an abuse of notation we identify $p_v'$ with its image $p_v'\in C^*_\theta (G'(F_v))$. We can therefore form
$$C^*_\theta(G(\A_{\rm fin}))=\sideset{}{'}{\bigotimes}_{v \in S_{\rm fin}(F)} (C^*_\theta(G_v), p_v)\quad\mbox{and}\quad C^*_\theta(G'(\A_{\rm fin}))=\sideset{}{'}{\bigotimes}_{v \in S_{\rm fin}(F)} (C^*_\theta(G'_v), p_v'),$$
sp that by \cite[Proposition 5.5]{GMS-24} form ideals in $C^*_r(G(\A_{\rm fin}))$ and $C^*(G'(\A_{\rm fin}))/\ker\Omega_{\A_{\rm fin}}^{r}$, respectively. In particular, the unitary duals of $C^*_\theta(G(\A_{\rm fin}))$ and $C^*_\theta(G'(\A_{\rm fin}))$ forms an open subset in the tempered unitary dual $\widehat{G(\A_{\rm fin})}_{\rm t}$ and the unitary dual of $C^*_\theta(G'(\A_{\rm fin}))$ forms a locally closed subset of $\widehat{G'(\A_{\rm fin})}$, respectively. The following theorem is concluced as above, using \cite[Proposition 5.5]{GMS-24}.

\begin{corollary}
\label{thetamoritaglobal}
Let $(G',G)$ be a type I dual pair over a number field $F$ with $G$ the smaller group. Assume that $(G',G)$ is either ortho-symplectic or unitary. Fix a global character $\chi : \A \to \C^1$ that is trivial on the subgroup $F$ and consider the global oscillator representation $\omega_\chi$ of $G'_{\A_{\rm fin}} \times G_{\A_{\rm fin}}$. The $\left (C^*_\theta(G'(\A_{\rm fin})), C^*_\theta(G(\A_{\rm fin})) \right )$-correspondence
$$\T_{\A_{\rm fin}}^r:=\sideset{}{'}{\bigotimes}_{v \in S_{\rm fin}(F)} (\T_v^r, \phi_v)$$
is a Morita equivalence implementing a homeomorphism from the unitary dual of $C^*_\theta(G(\A_{\rm fin}))$ to the unitary dual of $C^*_\theta(G'(\A_{\rm fin}))$ via Rieffel induction. In particular, the domain of the (finite) global theta correspondence in $\widehat{G(\A_{\rm fin})}_t$ is the open set $\widehat{C^*_\theta(G(\A_{\rm fin}))}$ and its range is the locally closed set $\widehat{C^*_\theta(G'(\A_{\rm fin}))}\subseteq \widehat{G'(\A_{\rm fin})}$.
\end{corollary}

\begin{remark}
The corresponding statement to Corollary \ref{thetamoritaglobal} for stable range dual pairs can be phrased as the fact that the (finite) global theta correspondence in $\widehat{G(\A_{\rm fin})}$ is the open set $\widehat{C^*_\theta(G(\A_{\rm fin}))}$ and its range is the locally closed set $\widehat{C^*_\theta(G'(\A_{\rm fin}))}\subseteq \widehat{G'(\A_{\rm fin})}$, where $C^*_\theta(G(\A_{\rm fin}))$ is now formed from the local, closed ideals $C^*_\theta(G(F_v))\subseteq C^*(G(F_v))$ spanned by the inner products $\langle \T_v,\T_v\rangle_{C^*(G(F_v))}$. 

\end{remark}


\section{Equal rank case revisited}
Here we revisit some of the results of \cite{mesland-sengun-ER} from the perspective of our current paper. Let $(G',G)$ be a Type I dual pair of equal rank over a non-archimedean local field. Then both the oscillator representations $\omega|_{\mid G'}$ and $\omega|_{\mid G}$ are tempered, and $\mathbb{S}$ is an $(\mathcal{S}(G'),\mathcal{S}(G))$ bimodule via the formulas \eqref{right-module} and \eqref{left-action}. The left $C^{*}(G')$ module structure of the bimodule $\T^{r}$ in fact factors through the quotient map $C^{*}(G')\to C^{*}_r(G')$. Thus, for equal rank pairs the bimodule $\T^{r}$ is in fact a $(C^{*}_{r}(G'),C^{*}_{r}(G))$-correspondence.

We now take a look at the global situation. We assume that $(G',G)$ is an equal rank dual pair over the number field $F$. For each $v\in S_{\rm fin}(F)$, \cite{mesland-sengun-ER} produces ideals 
$$C^*_\theta(G(F_v))\subseteq C^*_r(G(F_v))\quad \mbox{and}\quad C^*_\theta(G'(F_v))\subseteq C^*_r(G'(F_v)),$$
and the $(C^*_r(G'(F_v)),C^*_r(G(F_v))$-correspondence $\T_v^r$ constructed as above is in fact a Morita equivalence from $C^*_\theta(G(F_v))$ to $C^*_\theta(G'(F_v))$. In other words, $\T_v^r$ carries the structure of a right $C^*_r(G'(F_v))$-Hilbert $C^*$-module and a left $C^*_r(G(F_v))$-Hilbert $C^*$-module that are compatible as well as the right inner products $\langle \T_v^r,\T_v^r\rangle_{C^*_r(G'(F_v))}$ densely spanning $C^*_\theta(G'(F_v))$ and the left inner products ${}_{C^*(G_r(F_v))}\langle \T_v^r,\T_v^r\rangle$ densely spanning $C^*_\theta(G(F_v))$. This implies that the left action defines an isomorphism $C^*_\theta(G'(F_v))\simeq \mathbb{K}_{C^*_\theta(G(F_v))}(\T_v^r)$. In the considerations above, we note that \[p_v:=\langle\phi_v,\phi_v\rangle_{C^*_r(G(F_v))}\in C^*_\theta (G(F_v)),\quad\textnormal{and}\quad p_v':={}_{C^*_r(G(F_v))}\langle\phi_v,\phi_v\rangle\in C^*_\theta (G'(F_v)),\] so we can form
$$C^*_\theta(G(\A_{\rm fin}))=\sideset{}{'}{\bigotimes}_{v \in S_{\rm fin}(F)} (C^*_\theta(G_v), p_v)\quad\mbox{and}\quad C^*_\theta(G'(\A_{\rm fin}))=\sideset{}{'}{\bigotimes}_{v \in S_{\rm fin}(F)} (C^*_\theta(G'_v), p_v'),$$
that by \cite[Proposition 5.5]{GMS-24} form ideals in $C^*_r(G(\A_{\rm fin}))$ and $C^*_r(G'(\A_{\rm fin}))$, respectively. In particular, the unitary duals of $C^*_\theta(G(\A_{\rm fin}))$ and $C^*_\theta(G'(\A_{\rm fin}))$, respectively, form open subsets in the tempered unitary duals $\widehat{G(\A_{\rm fin})}_{\rm t}$ and $\widehat{G'(\A_{\rm fin})}_{\rm t}$, respectively. The following theorem is concluced as above, using \cite[Proposition 5.5]{GMS-24}.

\begin{theorem} 
\label{main-result-equal}  
Let $F$ be a number field set of finite places $S_{\rm fin}(F)$. Let $(G',G)$ be an equal rank dual pair over $F$. Fix a global character $\chi : \A \to \C^1$ that is trivial on the subgroup $F$ and consider the global oscillator representation $\omega_\chi$ of $G'_{\A_{\rm fin}} \times G_{\A_{\rm fin}}$. 
The $\left (C^*_\theta(G'(\A_{\rm fin})), C^*_\theta(G(\A_{\rm fin})) \right )$-correspondence
$$\T_{\A_{\rm fin}}^r:=\sideset{}{'}{\bigotimes}_{v \in S_{\rm fin}(F)} (\T_v^r, \phi_v)$$
is a Morita equivalence implementing a homeomorphism from the unitary dual of $C^*_\theta(G(\A_{\rm fin}))$ to the unitary dual of $C^*_\theta(G'(\A_{\rm fin}))$ via Rieffel induction. In particular, the domain of the (finite) global theta correspondence in $\widehat{G(\A_{\rm fin})}_t$ is the open set $\widehat{C^*_\theta(G(\A_{\rm fin}))}$ and its range is the locally closed set $\widehat{C^*_\theta(G'(\A_{\rm fin}))}\subseteq \widehat{G'(\A_{\rm fin})}_t$.
\end{theorem}

\section{Rallis inner product formula}
\label{sec:ralla}

In this section, we present a reformulation of the celebrated Rallis inner product formula in the framework of our $C^*$-algebraic approach to the global theta correspondence. Let $(G',G)$ be a Type I irreducible dual pair over a number field $F$. Let $\A$ denote the ring of adeles of $F$. Let $(\pi,V_\pi)$ cuspidal automorphic representation of $G(\A)$, that is, $\pi$ is an irreducible subrepresentation of the regular representation of $G(\A)$ on $L^2_0(G_F \backslash G(\A))$, where $L^2_0(G_F \backslash G(\A))$ is the cuspidal subspace of $L^2(G_F \backslash G(\A))$.

We fix a character $\chi : \A \to \C^1$ that is trivial on $F$. Let $\omega=\omega_\chi$ denote the associated global oscillator representation $\omega$ of $G'(\A) \times G(\A)$ realized on some $L^2(X_\A)$. Given cuspidal automorphic form $f \in V_\pi$  and a Schwartz function $\phi \in S(X_\A)$ on $X_\A$, function $\theta_\phi(f)$ on $G'(\A)$ given by the rule 
\begin{equation} 
\label{theta-kernel} 
g \mapsto \int_{G_F \backslash G(\A)} \left( \sum_{x \in X_F} \omega(gh)(\phi)(x) \right )  f(h) \mathrm{d}h 
\end{equation}
is well-defined and is an automorphic form on $G'(\A)$. The automorphic forms produced by the theta kernel method (\ref{theta-kernel}) are closely related to global theta correspondence. Indeed, let us denote by $\Theta(\pi)$ the representation of $G'(\A)$ afforded on the submodule of the space of automorphic forms on $G'(\A)$ spanned by
$$V(\pi):=\{ \theta_{\phi}(f) \mid f \in V_\pi, \ \ \phi \in S(X_\A) \}.$$
Let $\theta^{\mathrm{abs}}(\pi)$ denote the global theta lift of $\pi$, as discussed in the previous sections.

\begin{proposition} \label{abstract-equals-kernel}
Assume that $V(\pi)$ is non-trivial and consists of square-integrable automorphic forms on $G'(\A)$. Then $V(\pi)$ is irreducible and moreover, it is isomorphic to $\theta^{\mathrm{abs}}(\pi)$.
\end{proposition}

For proof, see \cite[Proposition 1.2]{Gelbart-Rogawski-Soudry} or \cite[Proposition 3.1]{Gan-AWS}. The following result of Rallis has been an important tool since its inception in the mid 1980's.

\begin{theorem}(Rallis Inner Product Formula) 
\label{RIPF} 
Let $\pi$ be a cuspidal automorphic representation of $G(\A)$. Fix 
$f_1,f_2 \in V_\pi$ and $\phi_1,\phi_2 \in S(X_\A)$. Assume that 
\begin{enumerate}
\item $\omega$ is integrable as a representation of $G(\A)$, 
\item the integral
\begin{equation} \label{norm-lift} \int_{G'_F \backslash G'(\A)} \theta_{\phi_1}(f_1)(g) \overline{\theta_{\phi_2}(f_2)(g)} dg
\end{equation}
converges.
\end{enumerate}
Then the integral (\ref{norm-lift}) equals
\begin{equation} \label{RIPF-1} \Big \langle \pi(\Psi_{\phi_1,\phi_2})(f_1), f_2 \Big \rangle 
\end{equation}
where $\Psi_{\phi_1,\phi_2} \in L^1(G(\A))$ is the matrix-coefficient function of $\omega$ associated to $\phi_1,\phi_2$.
\end{theorem}

This is proven by Rallis in \cite[Prop 1.1 and Remark 1.2]{Rallis-84} for the ortho-symplectic dual pairs. For general dual pairs, see Li \cite[Theorem 2.1]{Li-92}.  

We remark that (\ref{RIPF-1}) equals 
$$
\int_{G(\A)} \langle \omega(g)(\phi_1), \phi_2 \rangle \langle \pi(g)(f_1),f_2 \rangle dg.$$
If  $f_1,f_2,\phi_1,\phi_2$ are factorizable then we can express the above as a product
\begin{equation} \label{RIPF-2} \prod_{v} \int_{G_v} \langle \omega_v(s)(\phi_1),\phi_2 \rangle  \langle \pi_v(s)(f_1),f_2 \rangle ds
\end{equation}
where the product runs over the places $v$ of $F$ and $\omega_v$, $\pi_v$ are the local factors of $\omega$, $\pi$ respectively.
These local integrals, which lie at the heart of Li's method and our paper, can be expressed in terms of $L$-functions of the local representations (see e.g. \cite[Sec. 3]{Li-92}) and some authors reserve the name ``Rallis inner product formula'' for this ultimate expression in terms of $L$-functions.

\subsection{Reformulation of Rallis inner product formula} 
Assume that $(G',G)$ is in the stable range with $G$ the smaller group. Assume that $(G',G)$ is not the pair $(O_{2n,2n},Sp_{2n})$. Let  $\T_\A$ be the $(C^*(G'(\A)), C^*(G(\A)))$-correspondence described in our Theorem \ref{main-result-three}. Note that the space $S(X_\A)$ is densely contained in $\T_\A$. 

Let $\pi$ be a cuspidal automorphic representation of $G(\A)$. Assume that for any cuspidal automorphic form $f \in V_{\pi}$, the theta-kernel lift $\theta_{\phi}(f)$ is {\em square-integrable} for all $\phi \in S(X_\A)$. Then we can define the map 
\begin{equation} \label{intertwinerZ} \mathcal{Z} : \T_\A \otimes_{C^*(G(\A))} V_{\pi} \longrightarrow L^2(G'_F \backslash G'(\A)), \qquad  \mathcal{Z} (\phi \otimes f ) =  \theta_{\phi}(f).
\end{equation}
It is easy to see that this map is well-defined and is $G'(\A)$-equivariant. 

Given $\phi_1, \phi_2 \in S(X_\A)$ and $f_1,f_2 \in V_{\pi}$, we have, by construction, 
\begin{equation} \label{LHS} \big \langle \phi_1 \otimes f_1, \phi_2 \otimes f_2 \big \rangle := \big \langle \pi(\Psi_{\phi_1,\phi_2})(f_1), f_2 \big \rangle
\end{equation}
where $\Psi_{\phi_1,\phi_2}$ is the matrix coefficient of the oscillator representation. Therefore, asking for the intertwiner 
$\mathcal{Z}$ to be an isometry amounts to asking for the equality
\begin{equation} \label{RIPF_alt}
\big \langle \pi(\Psi_{\phi_1,\phi_2})(f_1), f_2 \big  \rangle = 
\big \langle \theta_{\phi_2}(f_1), \theta_{\phi_2}(f_2)  \big \rangle\subrangle{L^2(G'_F \backslash G'(\A))}, 
\end{equation}
which is precisely the Rallis inner product formula.

\begin{proposition} 
\label{rallisasisom}
Let $(G',G)$ is in the stable range with $G$ the smaller group. Assume that $(G',G)$ is not the pair $(O_{2n,2n},Sp_{2n})$. Let  $\T_\A$ be the $(C^*(G'(\A)), C^*(G(\A)))$-correspondence described in our Theorem \ref{main-result-three}. The Rallis inner product formula stated in Theorem \ref{RIPF} is equivalent to the intertwiner $\mathcal{Z}$ introduced in (\ref{intertwinerZ}) being an isometry.
\end{proposition}

Note that Theorem \ref{main-result-three} tells us that $\T_\A \otimes_{C^*(G(\A))} V_{\pi} \simeq \theta^{\mathrm{abs}}(\pi^*)$ where $\theta^{\mathrm{abs}}$ is the (abstract) global theta lifting as in Proposition \ref{abstract-equals-kernel}.


\end{document}